\documentclass{amsart}%
\usepackage{amsfonts}%
\usepackage{amsmath}%
\usepackage{amssymb}%
\usepackage{graphicx, epic,eepic, hyphenat, epsf}

\input xy
\xyoption{all}

\theoremstyle{definition}

\newtheorem{lemma}{Lemma}

\newtheorem{remark}{Remark}
\numberwithin{equation}{section}
\newcommand{\brak}[1]{\langle #1\rangle}

\DeclareMathOperator{\Hom}{Hom}

\DeclareMathOperator{\id}{Id}

\def\co{\colon\thinspace} 
\def\mf{\mathfrak}

\newskip\stdskip                      
\begin{document}

\title{On the $\mf{sl}(2)$ foam cohomology computations}
\author{Carmen Caprau}

\address{Department of Mathematics, California State University, Fresno, CA 93740 USA}
\email{ccaprau@csufresno.edu}

\date{}

\begin{abstract}
We show how to use Bar-Natan's `divide and conquer' approach to computation to efficiently compute the universal $\mf{sl}(2)$ dotted foam cohomology groups, even for big knots and links. We also describe a purely topological version of the $\mf{sl}(2)$ foam theory, in the sense that no dots are needed on foams. 
\end{abstract}
\maketitle

\section{Introduction}

In \cite{CC2} the author constructed the universal $\mf{sl}(2)$ link cohomology that uses dotted foams modulo local relations, and that categorifies the $\mf{sl}(2)$ polynomial. This was done in the spirit of Bar-Natan's~\cite{BN1} local approach to Khovanov homology on one side, and Khovanov's~\cite{Kh2} and Mackaay-Vaz's~\cite{MV} $\mf{sl}(3)$ link homology on the other side (see also~\cite{CC0, CC1} for a less general construction). The invariant of a link (or tangle) is a complex of graded free $\mathbb{Z}[i, a, h]$-modules, up to homotopy, where $a$ and $h$ are formal variables and $i$ is the primitive fourth root of unity. This theory corresponds to a certain Frobenius algebra structure defined on $\mathbb{Z}[i, a, h, X]/(X^2 -hX -a),$ and for the case of $a = h = 0$ it gives rise to an isomorphic version of the $\mf{sl}(2)$ Khovanov homology~\cite{BN0, Kh1}. The main improvement of the theory in~\cite{CC2} with respect to the original Khovanov homology and Bar-Natan's   work in~\cite{BN1} is the well-defined functorialily property with respect to link cobordisms relative to boundaries, with no sign indeterminacy (for details we refer the reader to~\cite{CC0, CC2}; we remark that similar construction and results to that in~\cite{CC0, CC1} were worked out by Clark, Morrison and Walker~\cite{CMW}, with no formal variables involved). Therefore it might be worthy to have a clear understanding of the $\mf{sl}(2)$ foam cohomology. 

Bar-Natan explained in~\cite{BN2} that his extension to tangle of the Khovanov homology via  cobordisms modulo relations yields an improvement in computational efficiency of the Khovanov homology groups. The purpose of this paper is to modify and adapt to foams the tools used in~\cite{BN2}, to obtain quick $\mf{sl}(2)$ dotted foam cohomology computations that otherwise would have taken more than a lifetime. We also give a description of a purely topological variant of this foam theory, variant in which no dots are required on cobordisms. 

In Section~\ref{sec:univKh} we briefly review the main results of the work done in~\cite{CC2}. We start Section~\ref{sec:fast computations} by providing the tools needed to obtain efficient computations and explaining the algorithm. Then we show that the tools we use yield an easy proof of the invariance of the foam cohomology under Reidemeister moves. After that, we consider the figure eight knot and apply the algorithm to it. Section~\ref{sec:no dots} contains the modified version of our dotted foam theory.

\section{Review of the universal $\mf{sl}(2)$ foam cohomology}\label{sec:univKh}

Given a tangle diagram $T$ with boundary $B,$ we construct a `formal complex' $[T]$ that lies in the additive category $\textit{Foams}_{/\ell}(B)$ whose objects are formally graded resolutions of $T$---called \textit{webs}---and whose morphisms are formal linear combinations of \textit{singular} cobordisms---called \textit{foams}---whose tops and bottoms are resolutions and whose side boundaries are $B \times I,$ modulo certain local relations. We explain these concepts below.

Each crossing of $T$ is replaced by one of the planar pictures below:
\[ \raisebox{-13pt}{\includegraphics[height=0.4in]{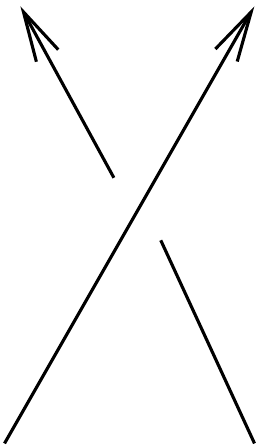}} \longrightarrow \raisebox{-13pt} {\includegraphics[height=0.4in]{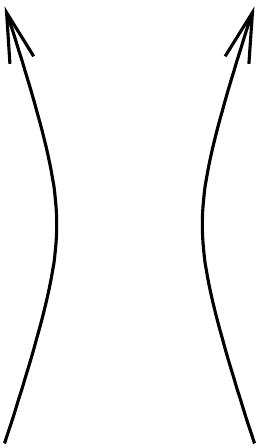}} \quad \text{and} \quad \raisebox{-13pt} {\includegraphics[height=0.4in]{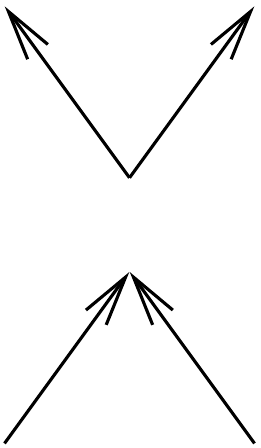}} \]
A diagram $\Gamma$ obtained by resolving all crossings of $T$ is a disjoint union of webs. A web with boundary $B$ is a planar graph $\Gamma$---properly embedded in a disk $\mathcal{D}^2$---with bivalent vertices near which the two incident edges are oriented either towards the vertex or away from it, and with univalent vertices that lie on the boundary of the disk $\mathcal{D}^2.$ Webs without vertices are also allowed.
 
The $\mf{sl}(2)$-link invariant is given by $P_2(D) = \sum_{\Gamma} \pm q^{\alpha(\Gamma)}\brak{\Gamma},$ where the sum is over all resolutions of $D,$ the exponents $\alpha(\Gamma)$ are determined by relations  
\begin{align}\raisebox{-13pt}{\includegraphics[height=0.4in]{poscrossing.pdf}} \,= \,q\,\,\,\,\,\raisebox{-13pt} {\includegraphics[height=0.4in]{orienres.pdf}} - q^2\,\,\,\,\,\raisebox{-13pt} {\includegraphics[height=0.4in]{singres.pdf}}\hspace{2cm}
\raisebox{-13pt}{\includegraphics[height=0.4in]{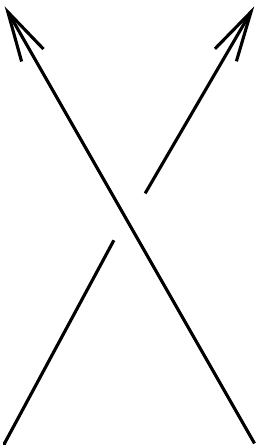}} \,= \,q^{-1}\,\raisebox{-13pt} {\includegraphics[height=0.4in]{orienres.pdf}} - q^{-2}\, \raisebox{-13pt} {\includegraphics[height=0.4in]{singres.pdf}}\end{align}
and the \textit{bracket polynomial} $\brak{\Gamma}$ associated to a closed web $\Gamma$ (its boundary $B$ is empty) is evaluated via the skein relations
\begin{align}
\brak{\raisebox{-5pt}{\includegraphics[height=0.2in]{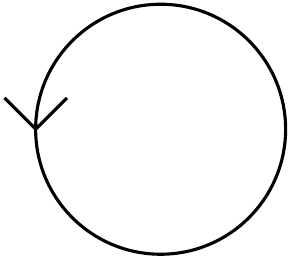}} \bigcup \Gamma} = (q + q^{-1}) \brak{\Gamma} =\brak{\raisebox{-5pt}{\includegraphics[height=0.2in]{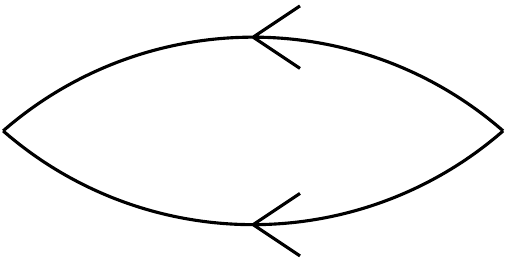}} \bigcup \Gamma} \\
\brak{\raisebox{-5pt}{\includegraphics[height=0.12in]{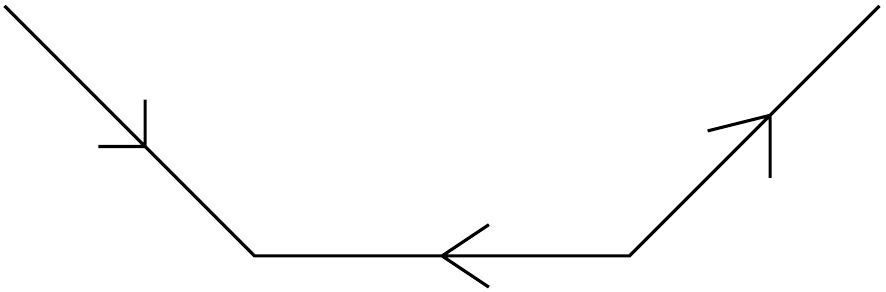}}} = 
\brak{\raisebox{-5pt}{\includegraphics[height=0.12in]{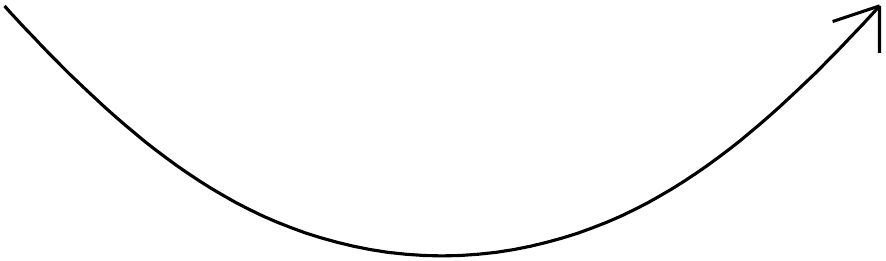}}}, \quad
\brak{\raisebox{-5pt}{\includegraphics[height=0.12in]{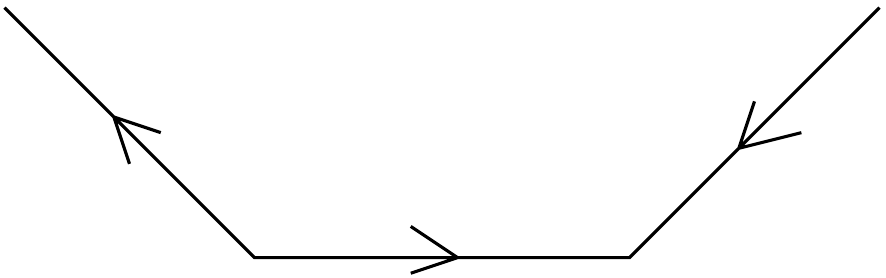}}} = 
\brak{\raisebox{-5pt}{\includegraphics[height=0.12in]{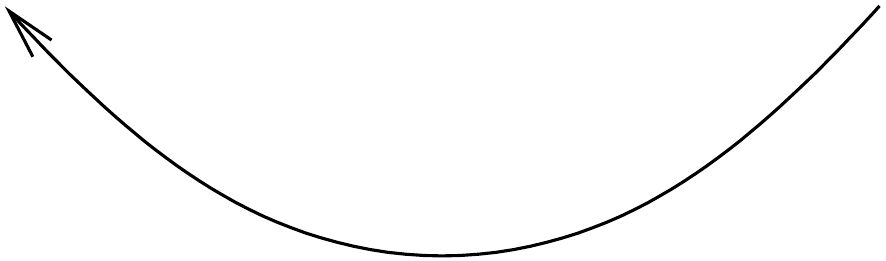}}}.\end{align}

A foam is an abstract cobordism between two webs $\Gamma_0$ and $\Gamma_1$ with boundary $B,$ regarded up to boundary-preserving isotopy. We read foams as morphisms from bottom to top by convention, and we compose them by placing one cobordism on top the other. Foams have \textit{singular arcs} (and/or \textit{singular circles}) where orientations disagree, and near which the facets incident with a given singular arc are compatibly oriented, inducing an orientation on that arc. Specifically, the orientation of singular arcs is as in the figure below, which shows examples of \textit{singular saddles}.
\[\includegraphics[height=.55in]{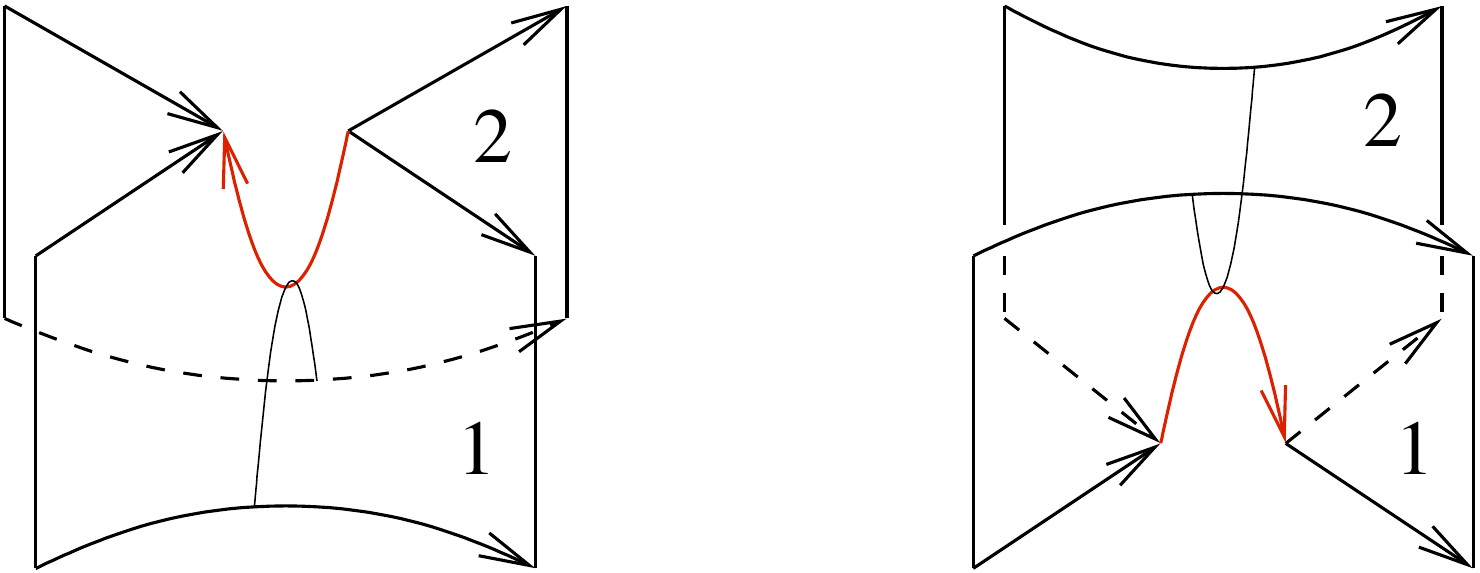}\]
For each singular arc of a foam, there is an ordering of foam's facets that are incident with it, in the sense that one of the facets is the \textit{preferred facet} for that singular arc. In some cases, we indicate facets' ordering by using labels 1 and 2, respectively, as shown above. In other cases, to prevent cluster, we draw a singular arc using a continuous red curve if the preferred facet for that singular arc is at its left in a given plane projection of a foam---where the concept of `left' or `right' is given by the orientation of the arc; otherwise, the singular arc is represented by a dashed red curve. Finally, foams can have dots that are allowed to move freely along the facet they belong to, but can't cross singular arcs.

We mod out the set of foams by the local relations $\ell$ = (2D, SF, S, UFO) below.
$$\xymatrix@R=2mm
{
\text{(2D)} \hspace{.2in}\raisebox{-5pt}{
\includegraphics[height=0.2in]{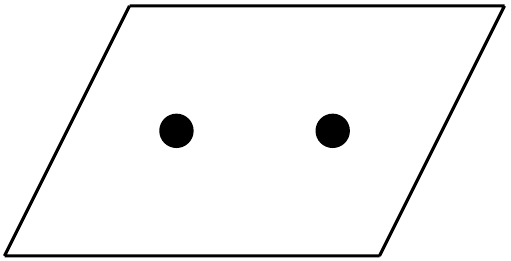}}=h\raisebox{-5pt}{\includegraphics[height=0.2in]{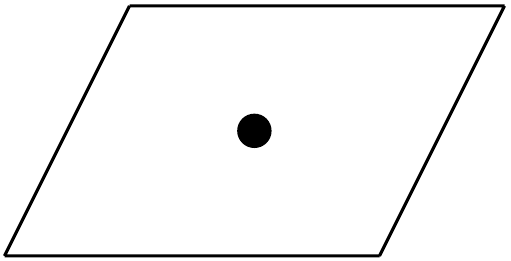}} + a\raisebox{-5pt}{\includegraphics[height=0.2in]{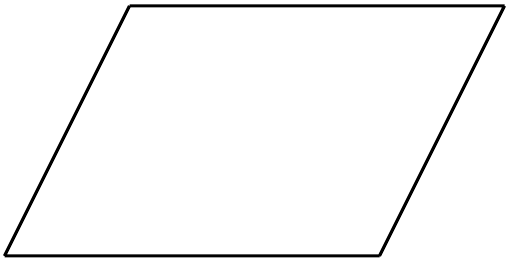}},\ \ \ \ \ \ \ 
\raisebox{-10pt}{\includegraphics[width=0.35in]{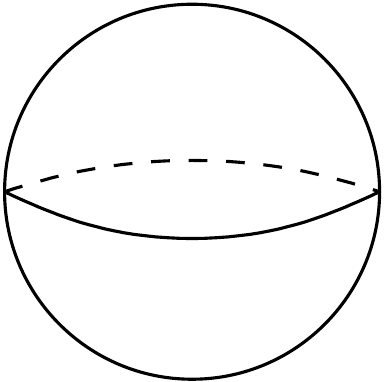}} =0,\quad
\raisebox{-10pt}{\includegraphics[width=0.35in]{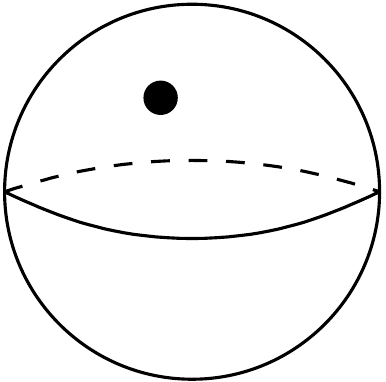}} =1
&\text{(S)} \\
\raisebox{-20pt}{
\includegraphics[height=0.6in]{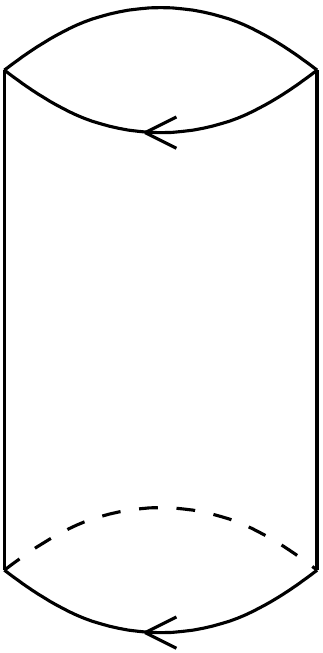}} =
\raisebox{-20pt}{
\includegraphics[height=0.6in]{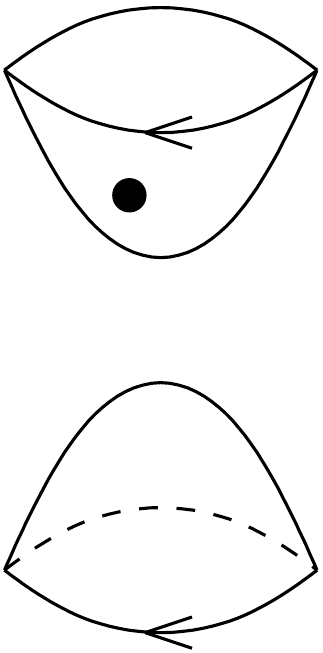}}+
\raisebox{-20pt}{
\includegraphics[height=0.6in]{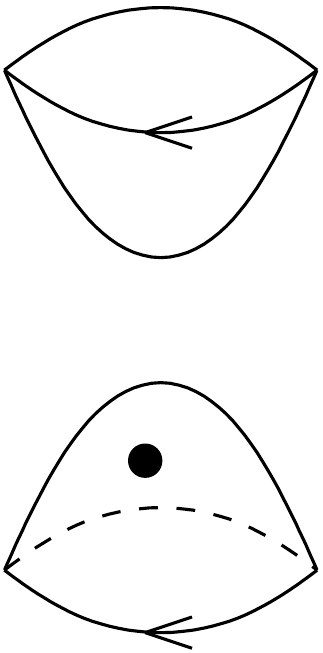}}-h
\raisebox{-20pt}{
\includegraphics[height=0.6in]{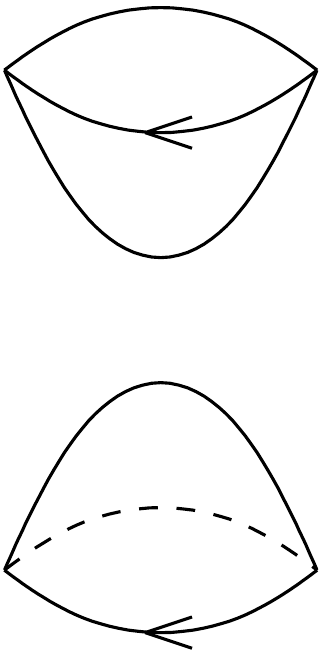}}
&\text{(SF)} \\
\raisebox{-10pt}{\includegraphics[width=0.5in]{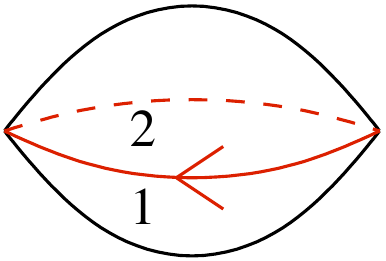}}=0=
\raisebox{-10pt}{\includegraphics[width=0.5in]{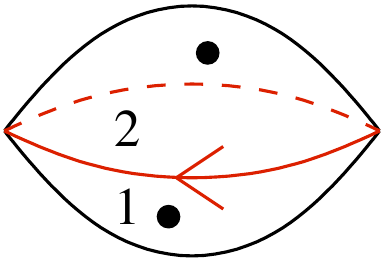}}\quad \text{and} \quad
\raisebox{-10pt}{\includegraphics[width=0.5in]{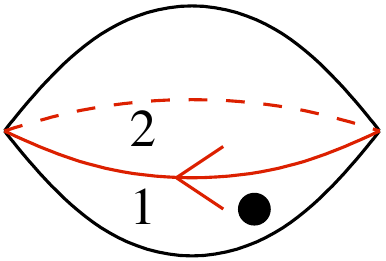}}=i= 
-  \raisebox{-10pt}{\includegraphics[width=0.5in]{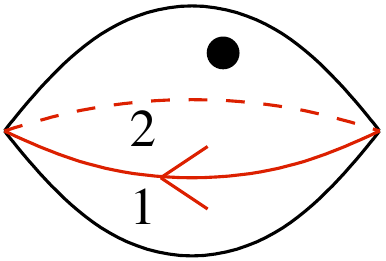}}
&\text{(UFO)} 
}$$
Relations $\ell$ imply the following useful \textit{curtain identities}:
\begin{align}
\raisebox{-13pt}{\includegraphics[height =0.5in]{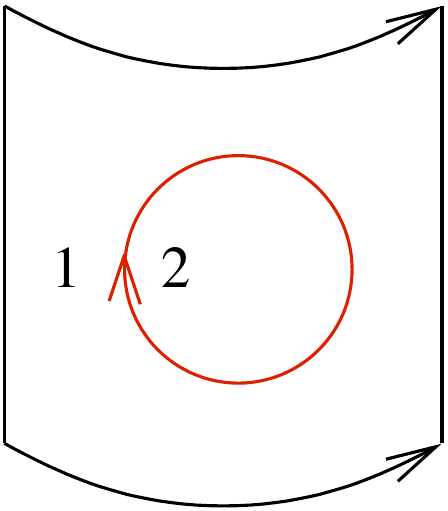}}&= i
\,\raisebox{-13pt}{\includegraphics[height=0.5in]{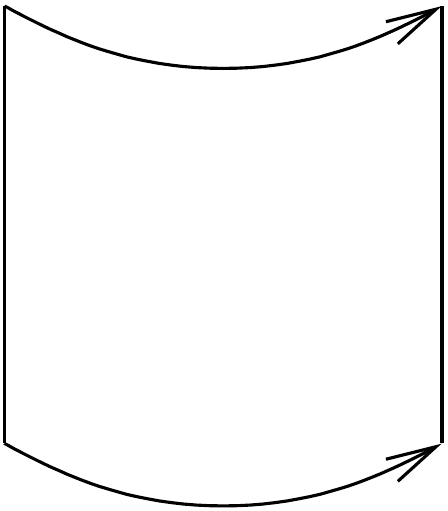}} \hspace{2cm}
\raisebox{-13pt}{\includegraphics[height =0.5in]{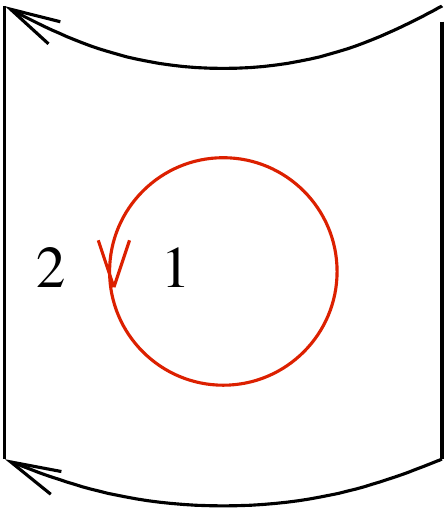}} =-i
\,\raisebox{-13pt}{\includegraphics[ height=0.5in]{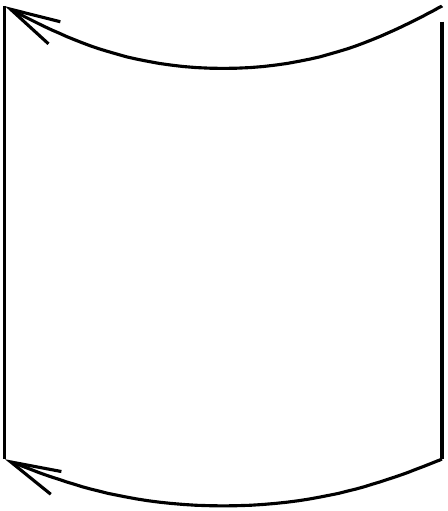}} \\
\raisebox{-18pt}{\includegraphics[height=0.6in]{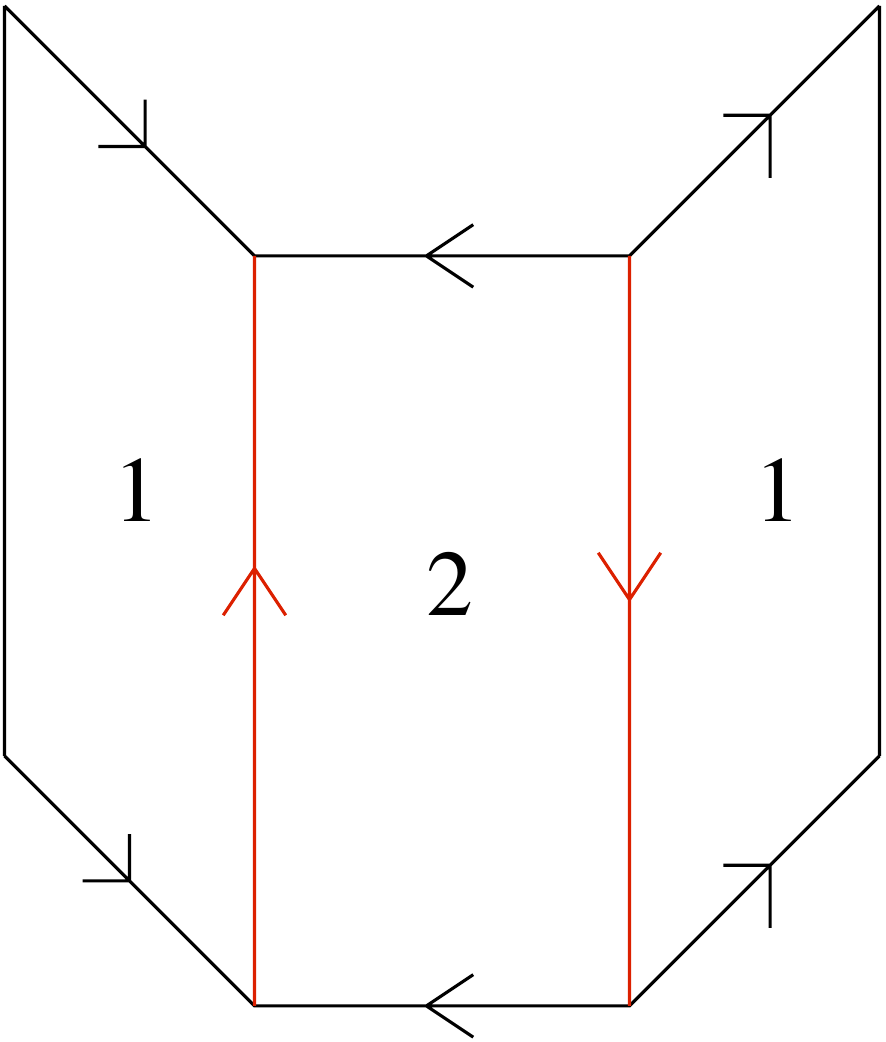}} &= -i \,
\raisebox{-18pt}{\includegraphics[height=0.6in]{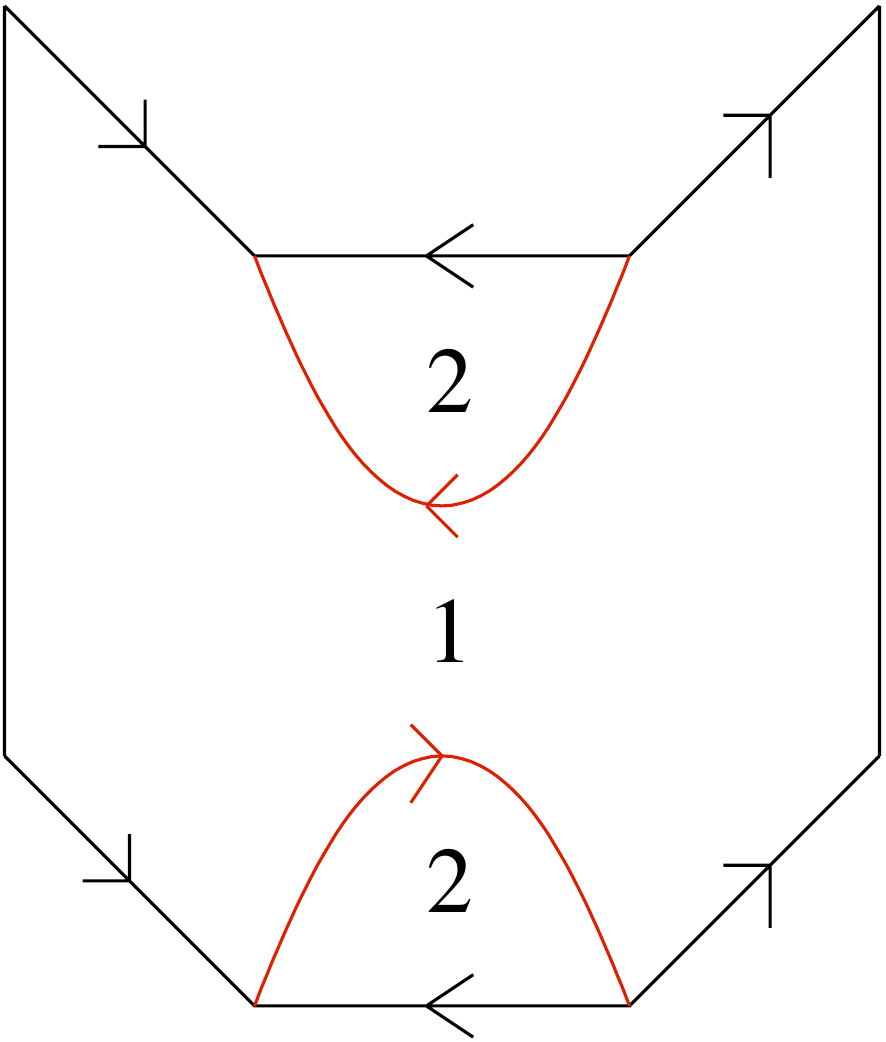}} \hspace{1cm}
\raisebox{-18pt}{\includegraphics[height=0.6in]{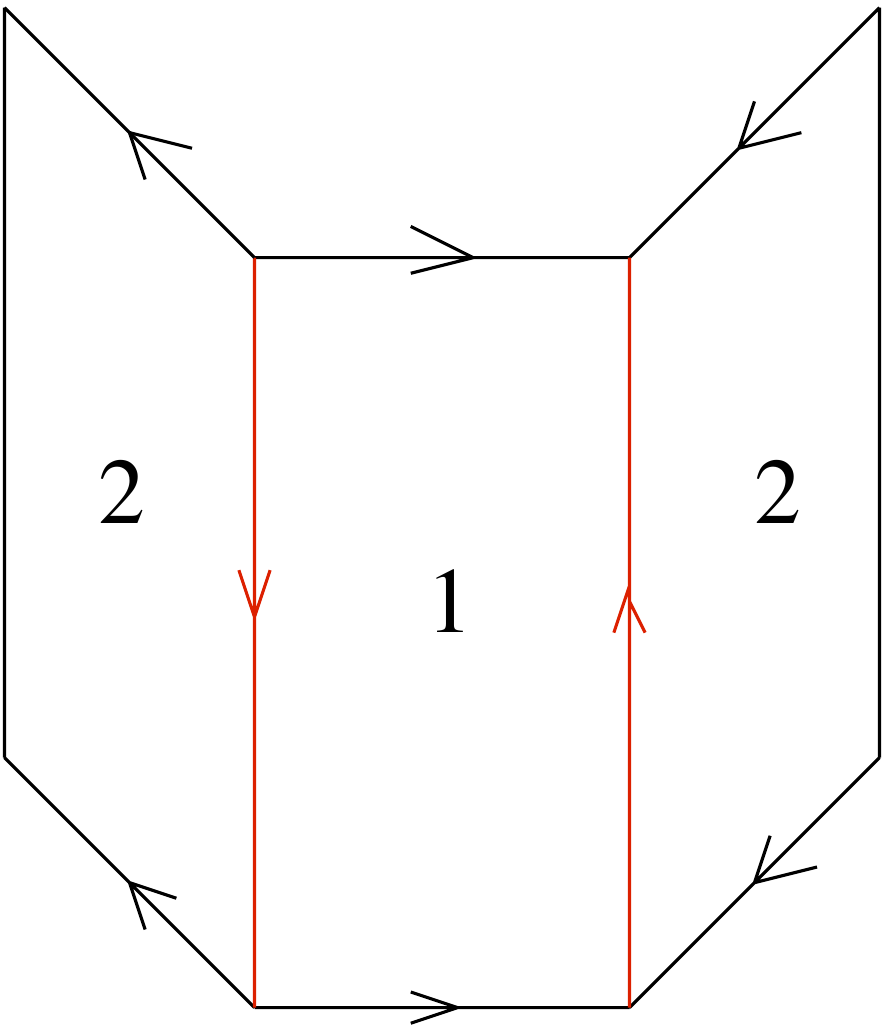}}= i\,
\raisebox{-18pt}{\includegraphics[height=0.6in]{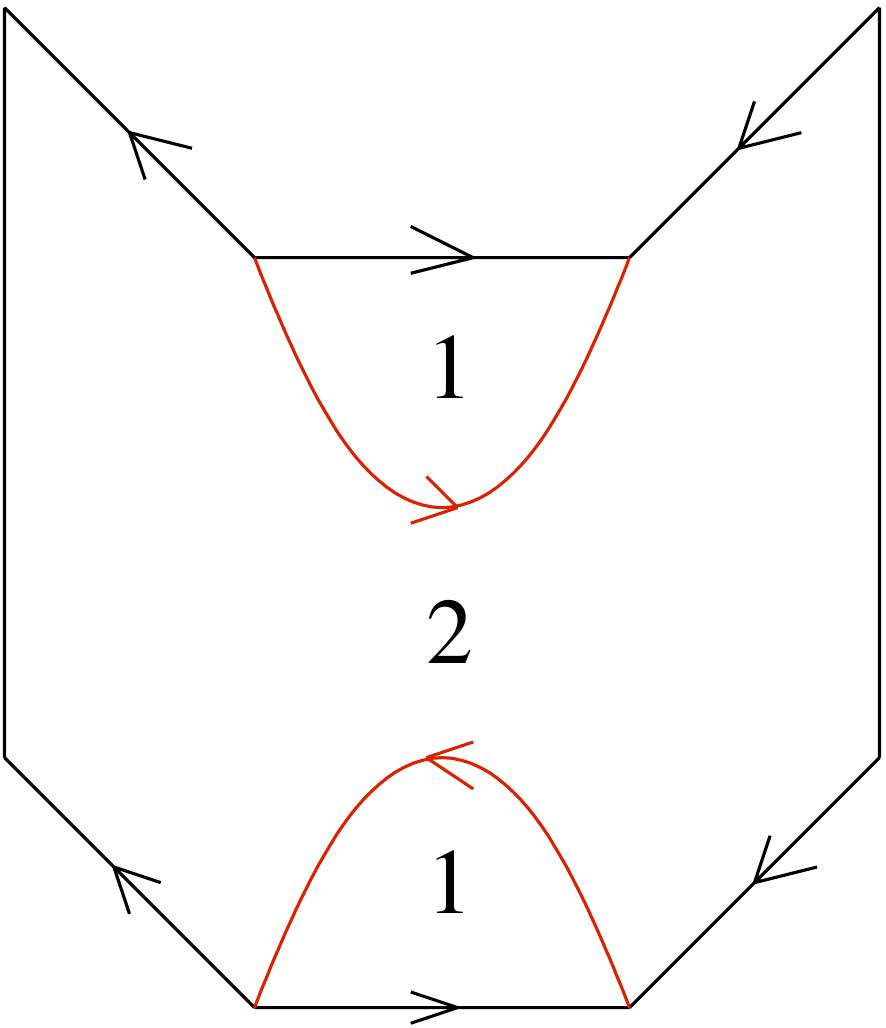}}
\end{align}
as well as the \textit{cutting neck} relation (CN), where the dots belong to the preferred facets:
$$\xymatrix@R=2mm{
\raisebox{-18pt}{\includegraphics[height=0.6in]{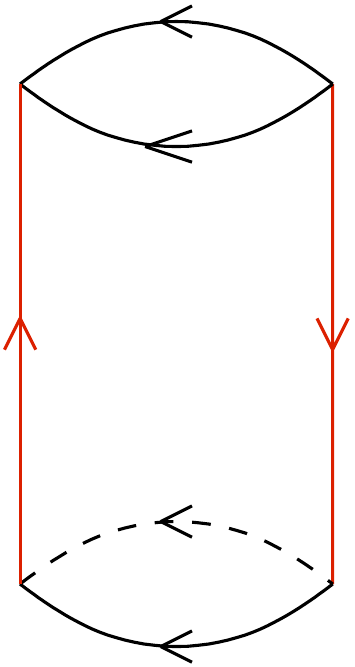}}= -i
\raisebox{-18pt}{\includegraphics[height=0.6in]{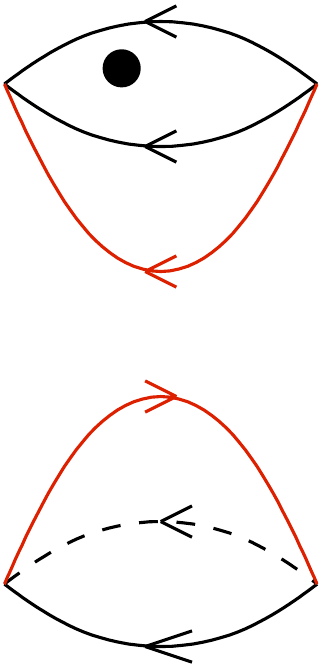}} -i
\raisebox{-18pt}{\includegraphics[height=0.6in]{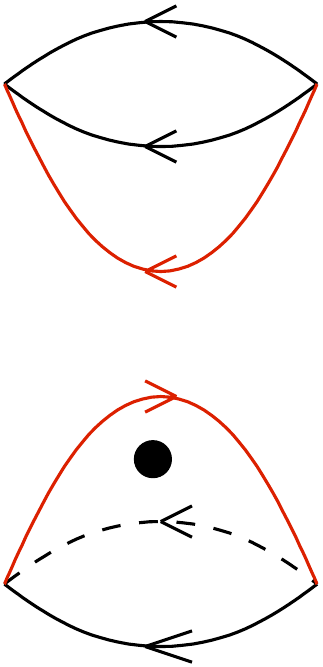}} + hi
\raisebox{-18pt}{\includegraphics[height=0.6in]{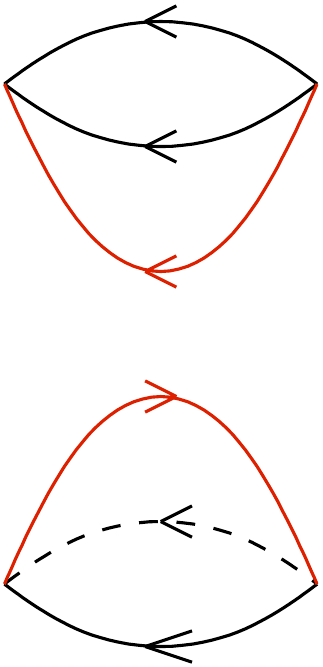}}
& \text{(CN)}
}$$
Local relations $\ell$ also give rules for \textit{exchanging dots} between neighboring facets, namely:
\begin{align}\label{exchanging dots}
\raisebox{-18pt}{\includegraphics[height=.5in]{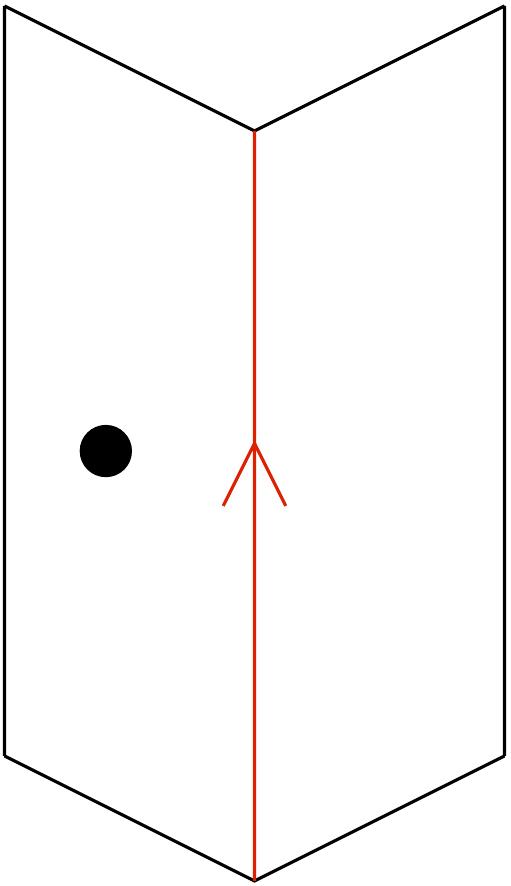}} +
\raisebox{-18pt}{\includegraphics[height=.5in]{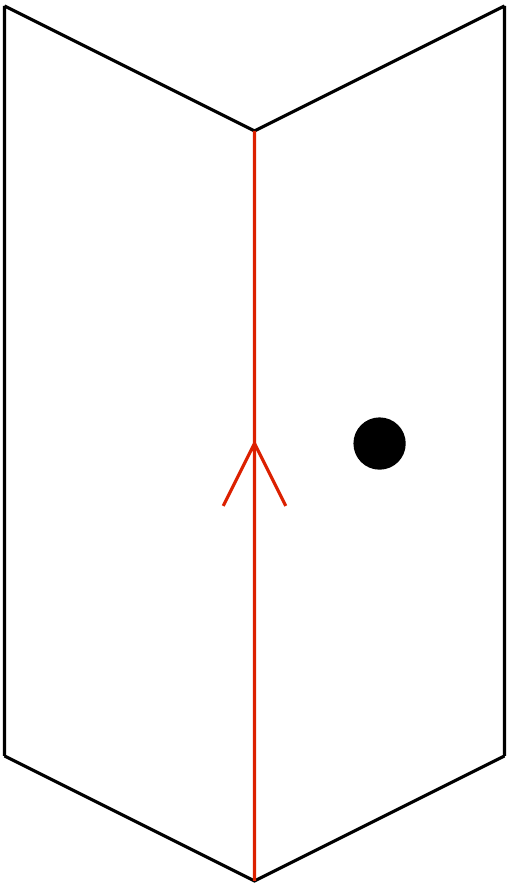}} = h\,\,\raisebox{-18pt}{\includegraphics[height=.5in]{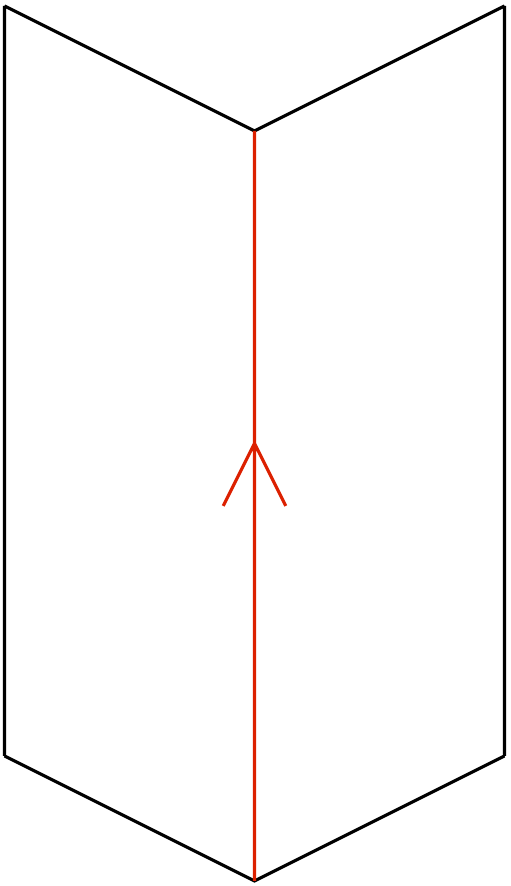}} \hspace{2cm}
\raisebox{-18pt}{\includegraphics[height=.5in]{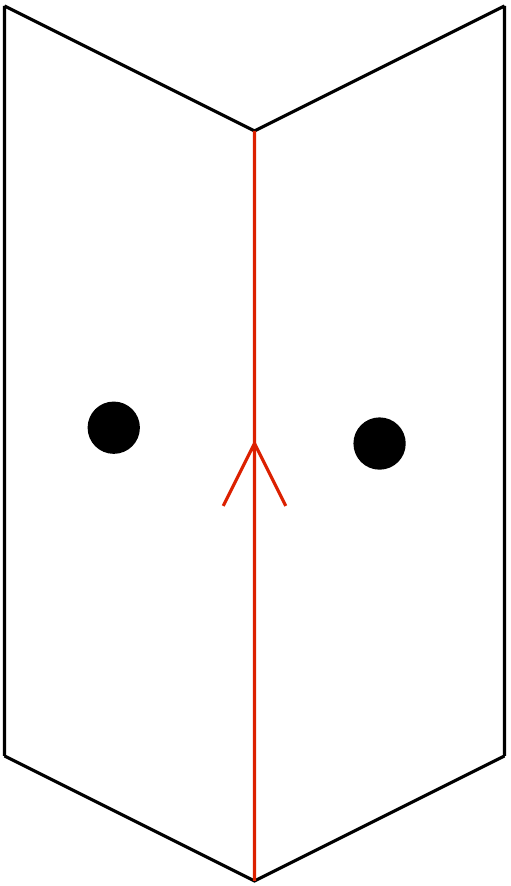}} = -a\, \,
 \raisebox{-18pt}{\includegraphics[height=.5in]{exch.pdf}}\end{align} 
 
 \begin{lemma}\label{lemma:isomorphisms}
Local relations $\ell$ imply that the following are isomorphisms in $\textit{Foam}_{/\ell}:$ 
\begin{align}\label{isomorphisms}
\raisebox{-25pt}{\includegraphics[height=.8in]{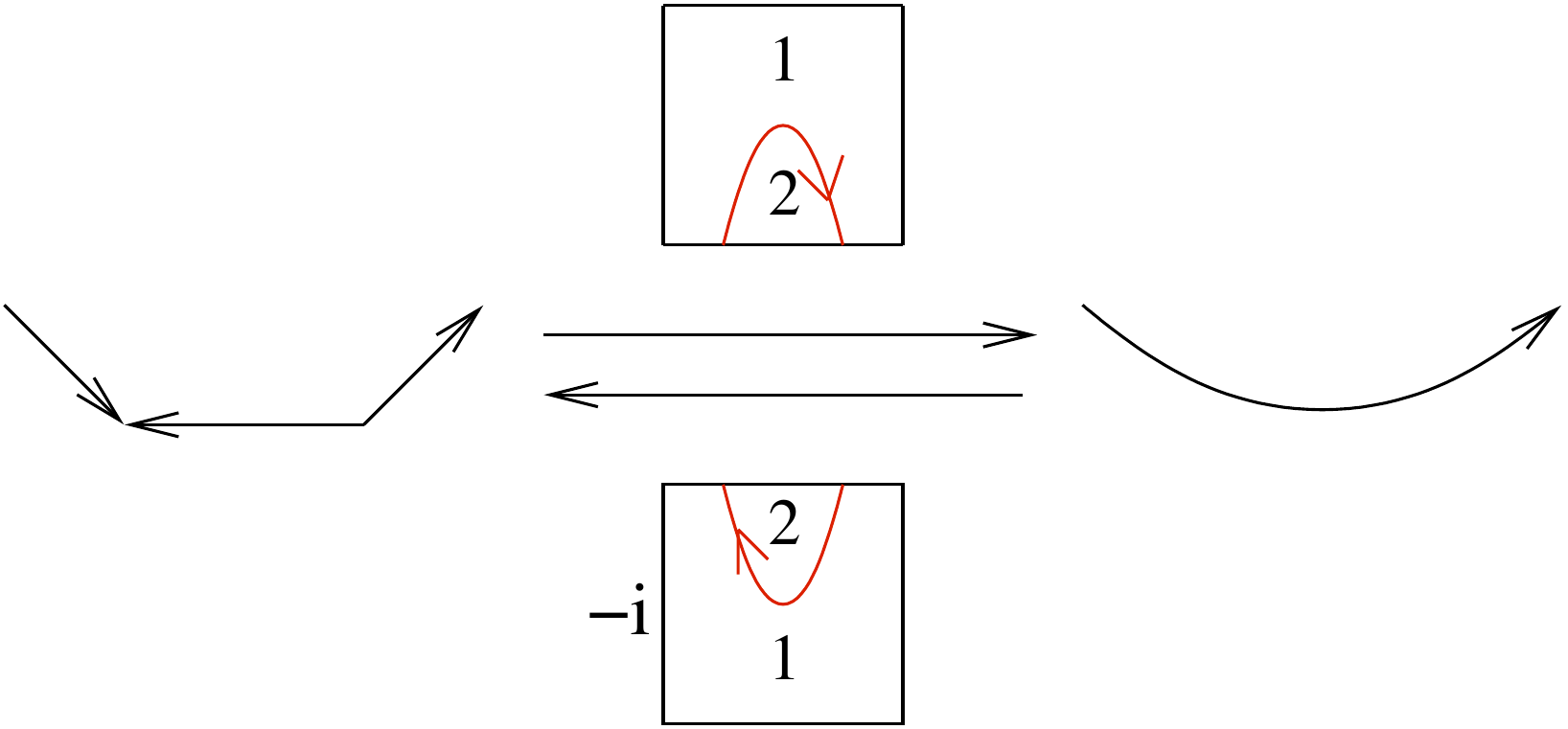}}
 \hspace{1.5cm}
\raisebox{-25pt}{\includegraphics[height=.8in]{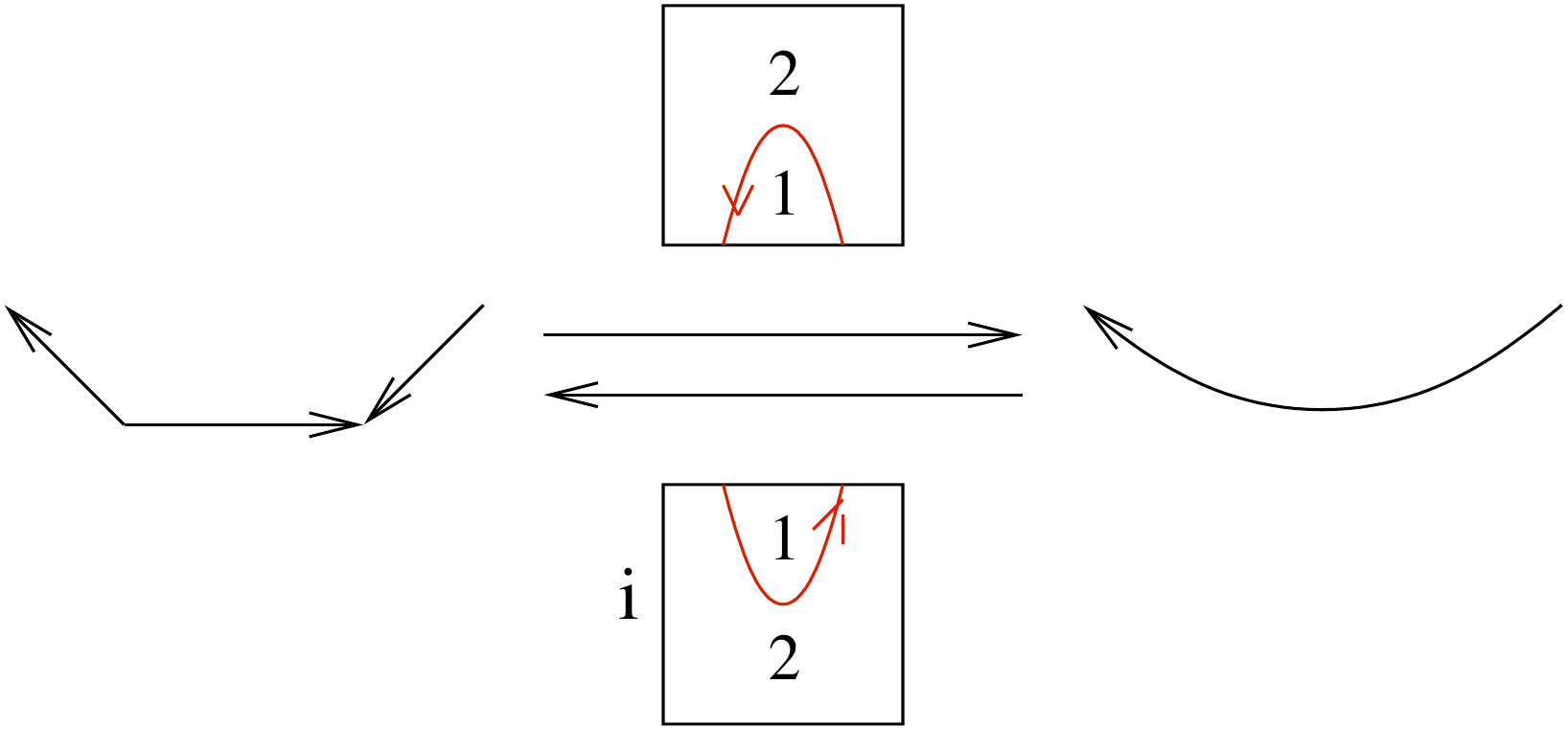}}\end{align}
\end{lemma}

 The formal complex $[T]$ is an object of the category \textit{Kof} := Kom(Mat($\textit{Foams}_{/\ell}(B)$)) of complexes of formal direct sums of objects in $\textit{Foams}_{/\ell}(B)$ and its cohomology $\mathcal{H}(T)=\oplus_{i,j \in \mathbb{Z}} \mathcal{H}^{i,j} (T)$ is a doubly graded invariant up to homotopy (here $i$ is the cohomological grading and $j$ is the polynomial grading).

\section{Efficient computations}\label{sec:fast computations}

In this section we apply to our setting Bar-Natan's `divide and conquer' approach to computations discussed in ~\cite{BN2}, to obtain an efficient algorithm for calculating the homology groups $\mathcal{H}^{i,j}(L)$ associated to a certain link diagram $L$, that otherwise would have taken a quite amount of time to evaluate. The key is to work locally, that is, to cut the link into subtangles, compute the invariant for each subtangle and finally assembly the obtained invariants into the invariant of $L,$ via the tensor product operation induced on formal complexes by the horizontal composition operation on the canopoly $\textit{Foams}_{/\ell}.$ To really obtain an improvement of computational efficiency, we simplify the complexes over the category $\textit{Foams}_{/\ell}$ before assembling, using the \textit{delooping} and \textit{Gaussian elimination} tools borrowed from~\cite{BN2} but adapted to our geometric picture and local relations.

\subsection {The tools and method}

Lemma~\ref{lemma:delooping} below is similar to~\cite[Lemma 4.1]{BN2} but uses our local relations, while Lemma~\ref{lemma:Gaussian elimination} is exactly~\cite[Lemma 4.2]{BN2} therefore we omit its proof.

\begin{lemma}\label{lemma:delooping}
(Delooping) Given an object of the form $S \cup \Gamma$ in the category $\textit{Foams}_{/\ell}$, where $\Gamma =  \raisebox{-4pt}{\includegraphics[height=0.2in]{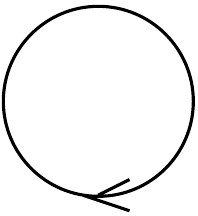}}$ or $\Gamma = \raisebox{-3pt}{\includegraphics[height=0.18in]{circle2sv.pdf}}$, it is isomorphic in Mat($\textit{Foams}_{/\ell}$) to the direct sum $S\{+1\} \oplus S\{-1\}$ in which $\Gamma$ is removed. This can be written symbolically as 
$$\raisebox{-4pt}{\includegraphics[height=0.2in]{unknot-clockwise.pdf}}\, \cong\, \emptyset \{+1\} \oplus \emptyset \{-1\} \quad \text{and} \quad \raisebox{-3pt}{\includegraphics[height=0.18in]{circle2sv.pdf}} \,\cong \, \emptyset \{+1\} \oplus \emptyset \{-1\}.$$
\end{lemma}
\begin{proof}
The desired isomorphisms are given in Figure~\ref{delooping}.
\begin{figure}
\includegraphics[height=1in]{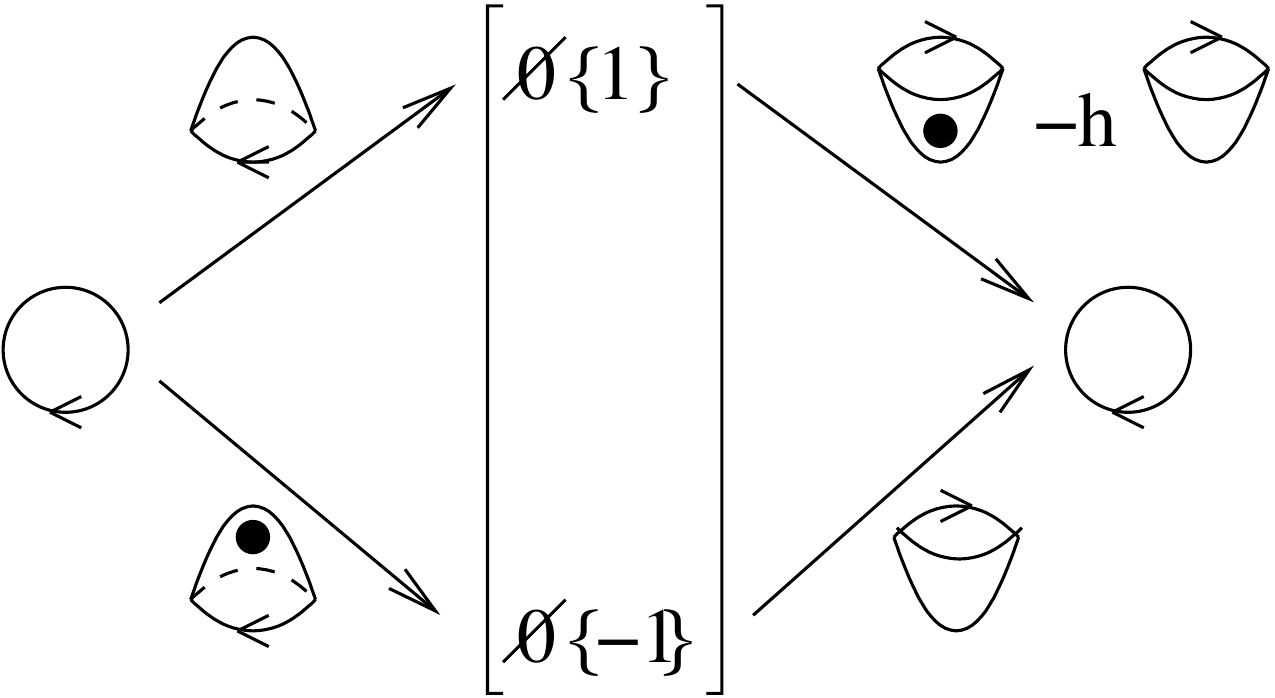} \qquad \includegraphics[height=1in]{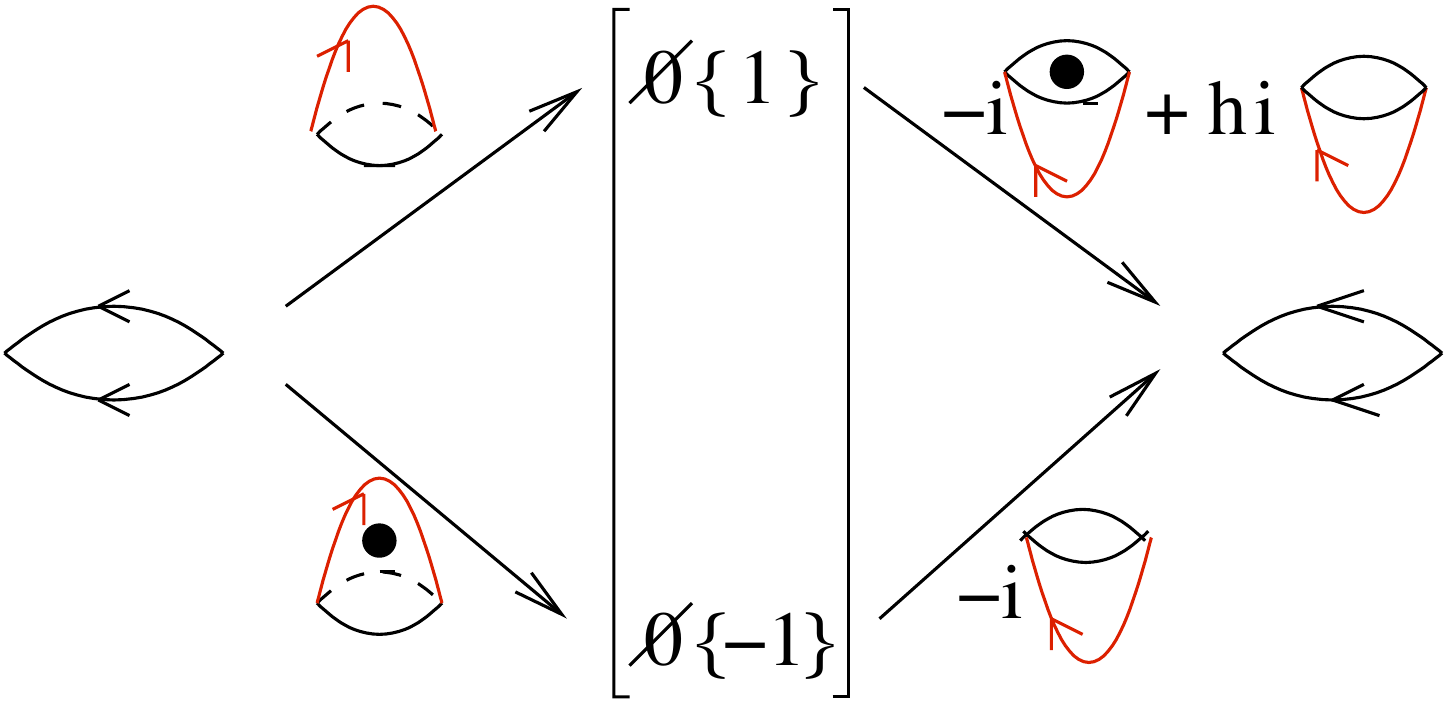}
\caption{Delooping isomorphisms} \label{delooping}
\end{figure}
Let us first show that $ \alpha = (\,\raisebox{-5pt}{\includegraphics[height=.25in]{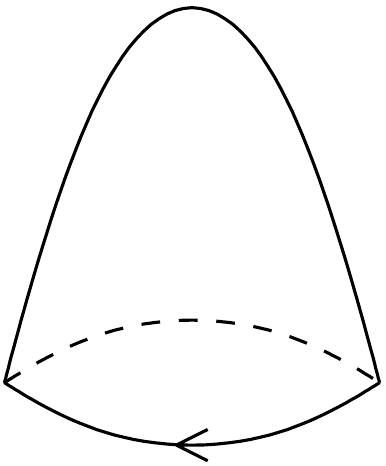}},\raisebox{-5pt}{\includegraphics[height=.25in]{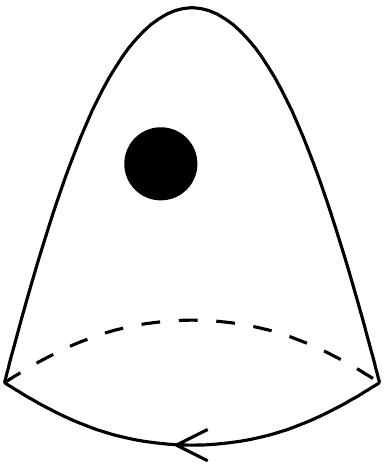}}\,)^t$ and $\beta = (\,\raisebox{-5pt}{\includegraphics[height=.25in]{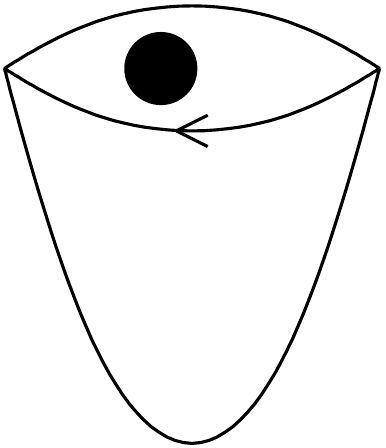}} -h \,\raisebox{-5pt}{\includegraphics[height=.25in]{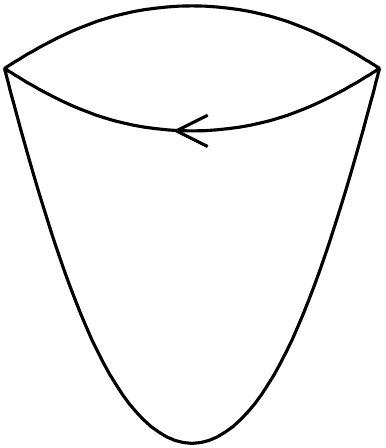}}, \raisebox{-5pt}{\includegraphics[height=.25in]{cupor.pdf}}\,)$ are mutually inverse isomorphisms. 
\[\beta \alpha =  \left ( \begin{array}{cc}\raisebox{-5pt}{\includegraphics[height=.25in]{cupordot.pdf}} -h \,\raisebox{-5pt}{\includegraphics[height=.25in]{cupor.pdf}}, & \raisebox{-5pt}{\includegraphics[height=.25in]{cupor.pdf}}\end{array} \right ) \circ \left ( \begin{array}{c}\raisebox{-5pt}{\includegraphics[height=.25in]{capor.pdf}} \vspace{.1cm} \\\raisebox{-5pt}{\includegraphics[height=.25in]{capordot.pdf}} \end{array} \right ) = \raisebox{-15pt}{\includegraphics[height=0.5in]{surgery1.pdf}} - h\raisebox{-15pt}{\includegraphics[height=0.5in]{surgery3}} + \raisebox{-15pt}{\includegraphics[height=0.5in]{surgery2.pdf}} \stackrel{(SF)}{=} \raisebox{-15pt}{\includegraphics[height=0.5in]{surgery.pdf}} = \id(\raisebox{-4pt}{\includegraphics[height=0.2in]{unknot-clockwise.pdf}})\]

\[ \alpha \beta = \left ( \begin{array}{c}\raisebox{-5pt}{\includegraphics[height=.25in]{capor.pdf}} \vspace{.1cm} \\\raisebox{-5pt}{\includegraphics[height=.25in]{capordot.pdf}} \end{array} \right ) \circ \left ( \begin{array} {cc} \raisebox{-5pt}{\includegraphics[height=.25in]{cupordot.pdf}} -h \,\raisebox{-5pt}{\includegraphics[height=.25in]{cupor.pdf}}, & \raisebox{-5pt}{\includegraphics[height=.25in]{cupor.pdf}} \end{array} \right ) = \left ( \begin{array}{cc} \raisebox{-5pt}{\includegraphics[width=0.25in]{sphd.pdf}} - h \,\raisebox{-5pt}{\includegraphics[width=0.25in]{sph.pdf}} & \raisebox{-5pt}{\includegraphics[width=0.25in]{sph.pdf}} \vspace{.15cm} \\ \raisebox{-5pt}{\includegraphics[width=0.25in]{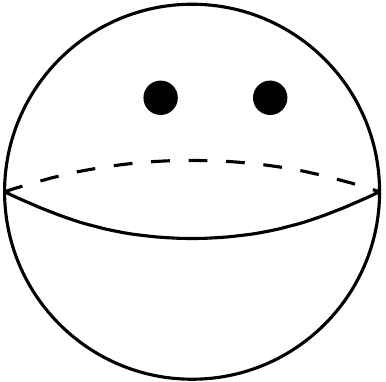}} - h\, \raisebox{-5pt}{\includegraphics[width=0.25in]{sphd.pdf}} &\raisebox{-5pt}{\includegraphics[width=0.25in]{sphd.pdf}}     \end{array} \right) \stackrel {(S), (2D)}{=} \left( \begin{array}{cc} 1 & 0 \\ 0 & 1 \end{array} \right ). \]

In the same manner, one can use relations (UFO) and (CN) to verify that $(\,\raisebox{-5pt}{\includegraphics[height=.25in]{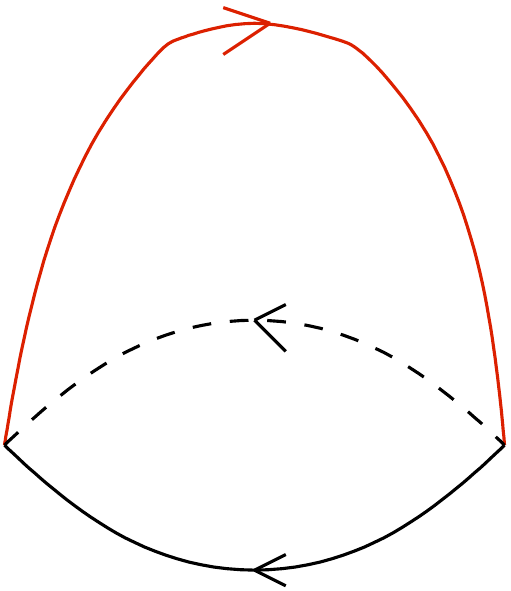}}, \raisebox{-5pt}{\includegraphics[height=.25in]{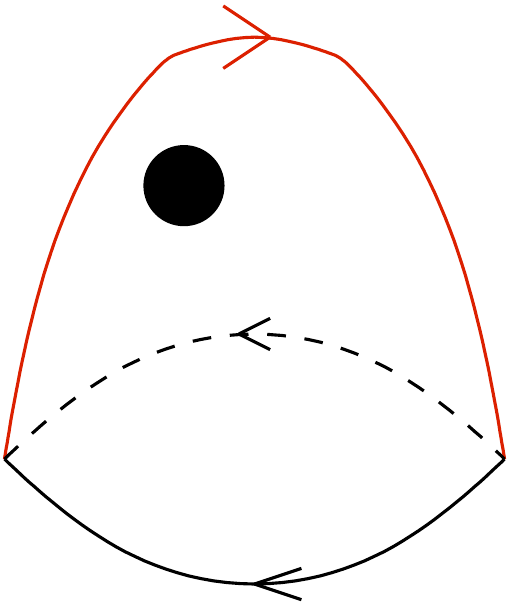}}\,)^t$ and $(\,-i\, \raisebox{-5pt}{\includegraphics[height=.25in]{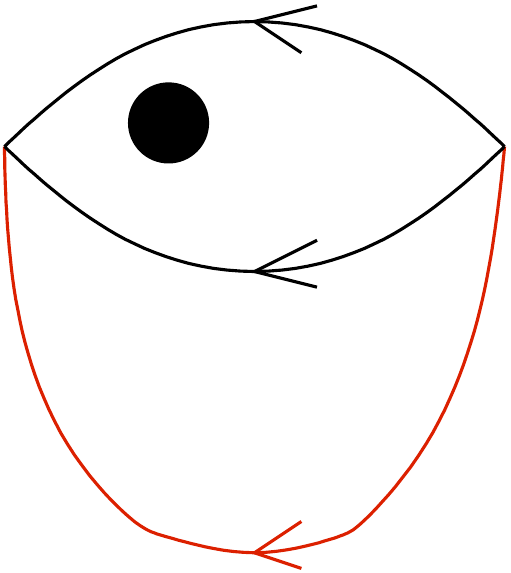}} +  hi\,\raisebox{-5pt}{\includegraphics[height=.25in]{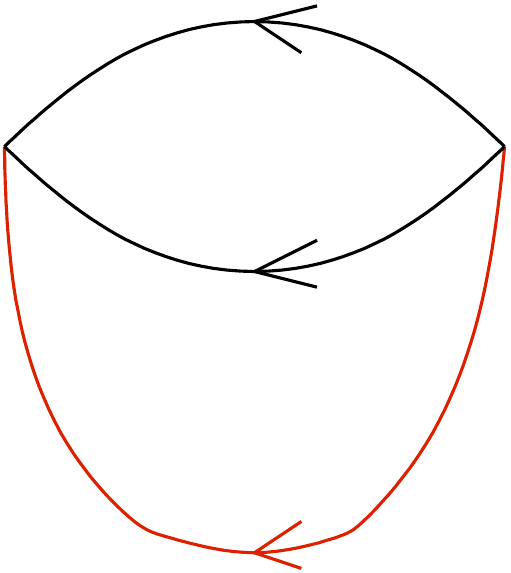}}, -i\, \raisebox{-5pt}{\includegraphics[height=.25in]{cupsa.pdf}}\,)$ are also mutually inverse isomorphisms.
\end{proof}

\begin{lemma} (Gaussian elimination for complexes) \label{lemma:Gaussian elimination}
If $\phi : b_1 \rightarrow b_2$ is an isomorphism in some additive category $\mathcal{C}$, then the complex segment in Mat($\mathcal{C}$) 
\[... \left [C \right ] \stackrel{\left (\begin{array}{c} \alpha \\ \beta \end{array}\right)}{\longrightarrow} \left [ \begin{array}{c} b_1 \\ D \end{array} \right ] \stackrel{ \left (\begin{array}{cc} \phi &\delta \\ \gamma &\epsilon \end{array} \right)}{\longrightarrow} \left [ \begin{array}{c} b_2 \\ E \end{array} \right ] \stackrel{\left( \begin{array}{cc} \mu & \nu \end{array}\right )}{\longrightarrow} \left [ F \right ]...\]
is isomorphic to the complex segment 
\[... \left [C \right ] \stackrel{\left (\begin{array}{c} 0 \\ \beta \end{array}\right)}{\longrightarrow} \left [ \begin{array}{c} b_1 \\ D \end{array} \right ] \stackrel{ \left (\begin{array}{cc} \phi & 0 \\ 0 &\epsilon - \gamma \phi ^{-1} \delta \end{array} \right)}{\longrightarrow} \left [ \begin{array}{c} b_2 \\ E \end{array} \right ] \stackrel{\left( \begin{array}{cc} 0 & \nu \end{array}\right )}{\longrightarrow} \left [ F \right ]...\]
The later is the direct sum of the contractible complex
$ 0 \longrightarrow b_1 \stackrel{\phi}{\longrightarrow} b_2 \longrightarrow 0 $
and the complex segment 
$... \left [C \right ] \stackrel{\left (\beta \right)}{\longrightarrow} \left [D \right] \stackrel{\left (\epsilon - \gamma \phi ^{-1} \delta \right)}{\longrightarrow} \left [E \right ] \stackrel{\left(\nu \right )}{\longrightarrow} \left [ F \right ]... \,$. Therefore, the first and last complex segments are homotopy equivalent.
\end{lemma}

Whenever an object in some formal complex $\Lambda \in \textit{Kof}_{/h}$ contains an oriented loop or a basic closed web with two vertices, we remove it using Lemma~\ref{lemma:delooping}, where $\textit{Kof}_{/h}$ is the homotopy subcategory of \textit{Kof}. Then we use Lemma~\ref{lemma:Gaussian elimination} to cancel all isomorphisms in the resulting complex.

\subsection{Examples}\label{sec:examples}

Let us first see how one can use the tools described in Section~\ref{sec:fast computations} to show the homotopy invariance of the complex $[T]$ under Reidemeister moves. For this, one has to compute and simplify the complexes corresponding to each side of a certain Reidemeister move, to obtain the same result for both sides.

\subsection*{\textbf{Reidemeister I}} Consider diagrams $D=\raisebox{-13pt}{\includegraphics[height=0.4in]{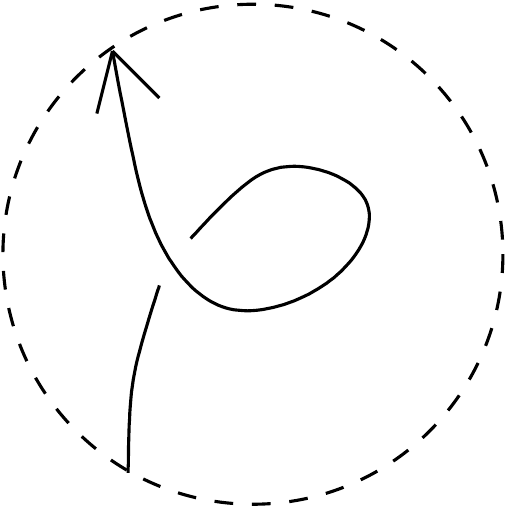}}\,\,\text{and}\,
D'=\raisebox{-13pt}{\includegraphics[height=0.4in]{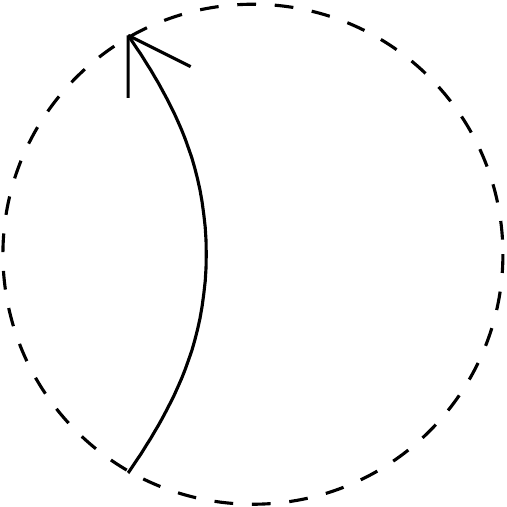}}.$
The complex $ [D]: \,\, 0 \longrightarrow \underline{\left [\raisebox{-8pt}{\includegraphics[height=0.3in]{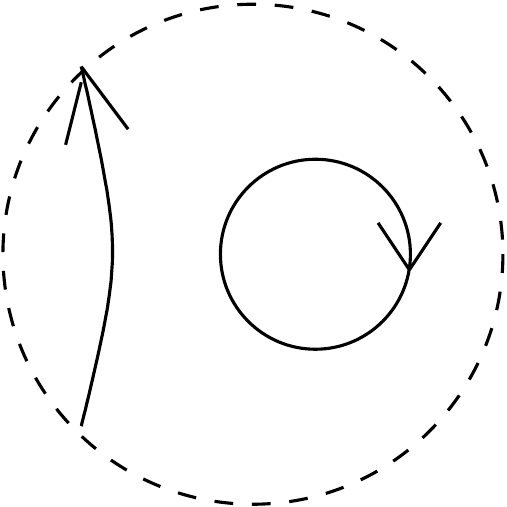}}\right ]\{-1\}}\stackrel{\raisebox{-8pt}{\includegraphics[height=0.3in]{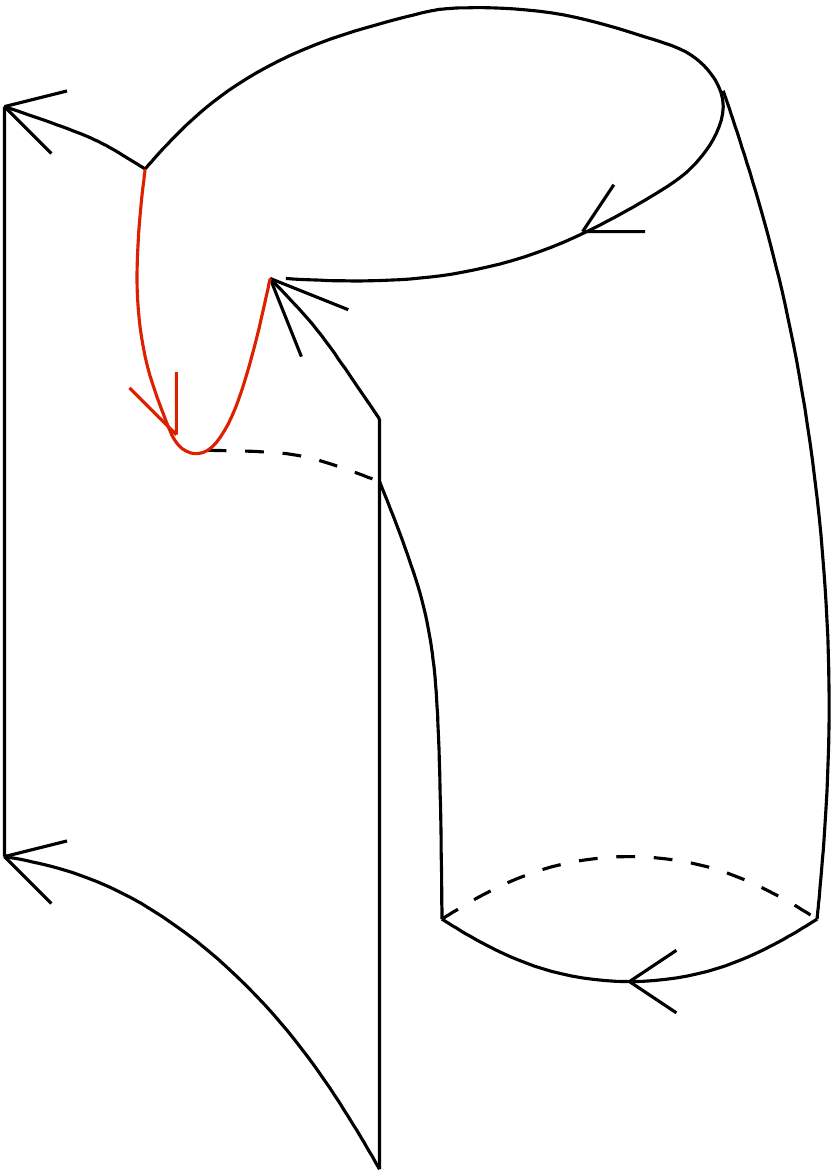}}}{\longrightarrow}\left [ \raisebox{-8pt} {\includegraphics[height=0.3in]{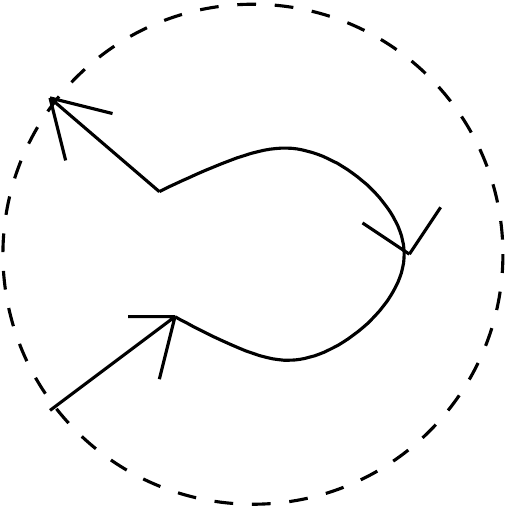}}\right ]\{-2\} \longrightarrow 0$
is isomorphic  in $\textit{Kof}_{/h}$ to   
$$ 0 \longrightarrow \underline{\left [\begin{array}{c} \raisebox{-8pt}{\includegraphics[height=0.3in]{reid1-1.pdf}} \{-2\}\vspace{.1cm} \\  \raisebox{-8pt}{\includegraphics[height=0.3in]{reid1-1.pdf}}\{0\} \end{array} \right] }\stackrel{\left ( \begin{array}{cc}\raisebox{-8pt}{\includegraphics[height=0.3in]{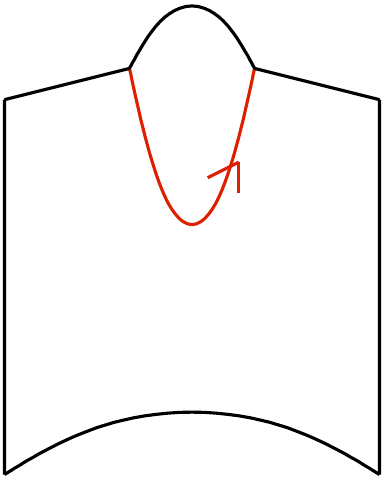}} & \raisebox{-8pt}{\includegraphics[height=0.3in]{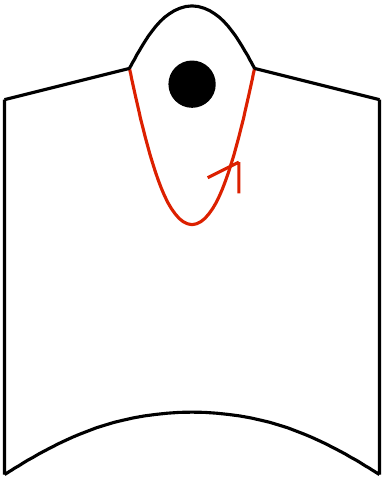}} \end{array} \right )}{\longrightarrow}\left [ \raisebox{-8pt} {\includegraphics[height=0.3in]{reid1-3.pdf}}\right ]\{-2\} \longrightarrow 0.$$

Note that the underlined objects above are at the cohomological degree $0.$ The later complex decomposes into contractible complex (its differential is an isomorphism)
$$ 0 \longrightarrow \underline{\left [ \raisebox{-8pt}{\includegraphics[height=0.3in]{reid1-1.pdf}}\right ] \{-2\}}\stackrel{\left ( \raisebox{-8pt}{\includegraphics[height=0.3in]{isom-8.pdf}}\right )}{\longrightarrow} \left [\raisebox{-8pt} {\includegraphics[height=0.3in]{reid1-3.pdf}}\right ]\{-2\} \longrightarrow 0 \quad \text{and}$$
$$0 \longrightarrow \underline{\,\left [\raisebox{-8pt}{\includegraphics[height=0.3in]{reid1-1.pdf}} \,\right ]}\longrightarrow 0.$$
Hence, complexes $[\,\raisebox{-8pt}{\includegraphics[height=0.3in]{lkink.pdf}}\,]$ and $[\,\raisebox{-8pt}{\includegraphics[height=0.3in]{reid1-1.pdf}}\,]$ are homotopy equivalent.

\subsection*{\textbf{Reidemeister II}}
Consider diagrams $D=\raisebox{-13pt}{\includegraphics[height=0.4in]{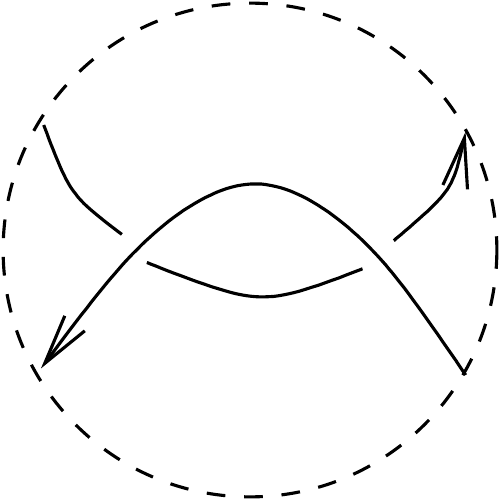}}$ and 
$D'=\raisebox{-13pt}{\includegraphics[height=0.4in]{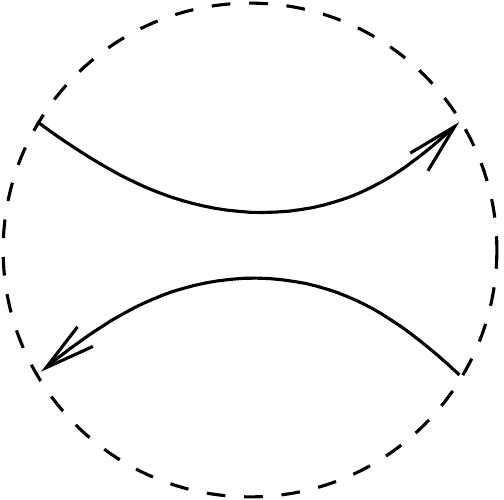}}.$ The formal complex $[D]$ is the double complex given below, which is the tensor product of the formal complexes associated with the two crossings in $D.$

$$D = \raisebox{-8pt}{\includegraphics[height=0.3in]{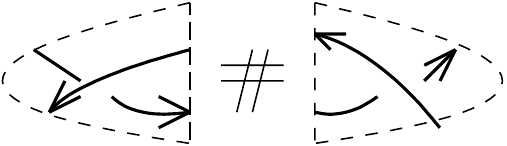}} \qquad \raisebox{-85pt}{\includegraphics[height=2.5in]{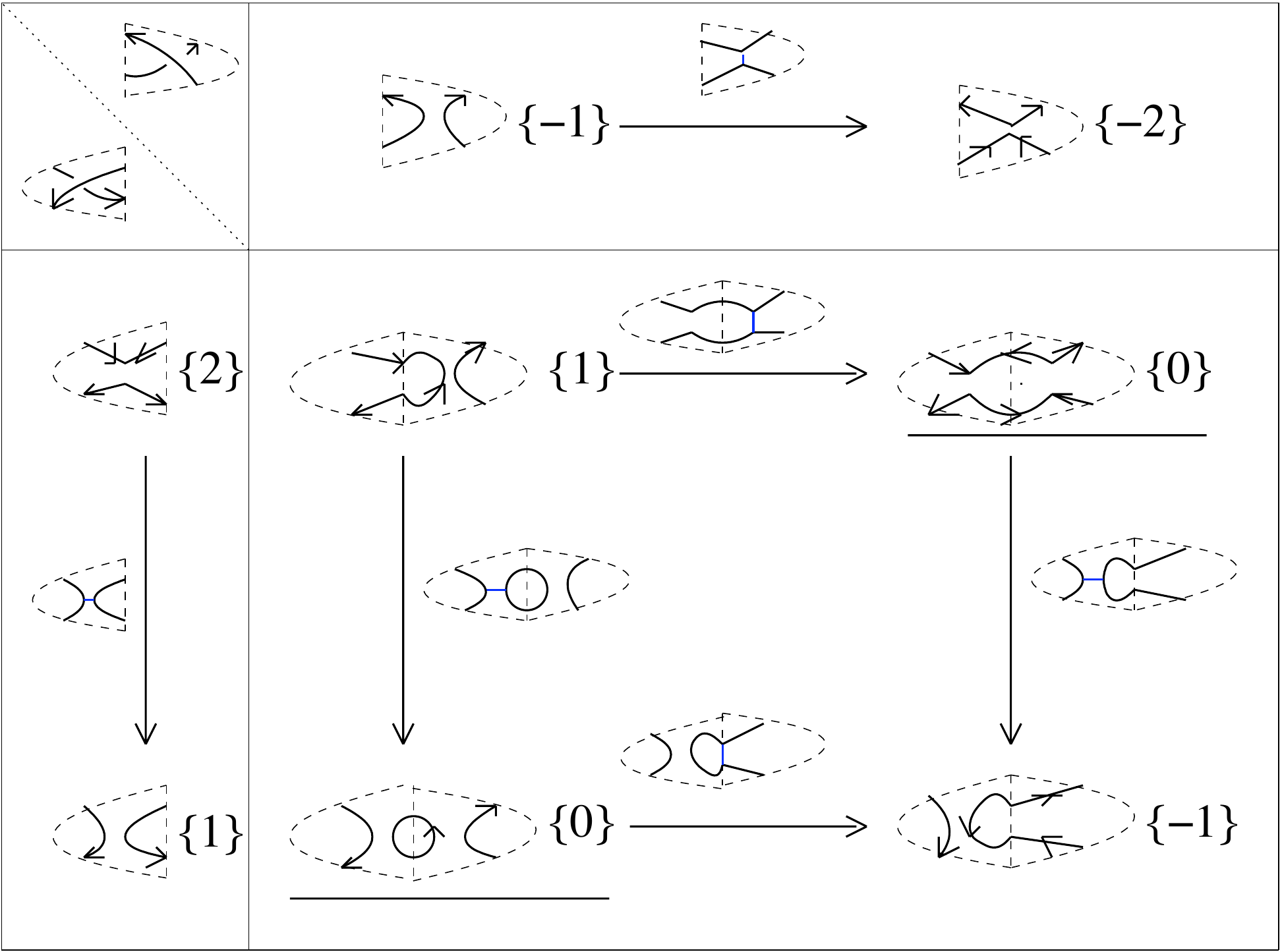}}$$

$$ \includegraphics[height=1.1in]{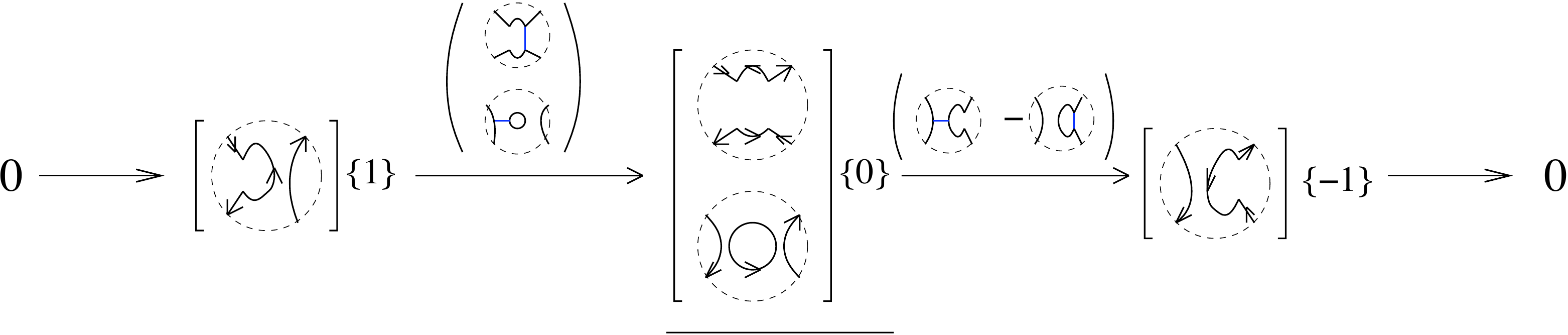}$$

The morphism \raisebox{-8pt}{\includegraphics[height=0.3in]{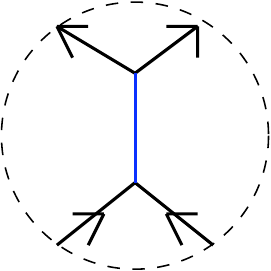}} is the `singular saddle' with domain \raisebox{-8pt}{\includegraphics[height=0.3in]{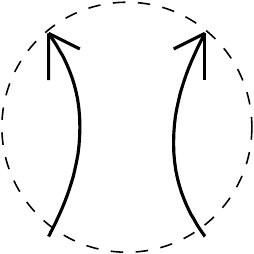}} and range \raisebox{-8pt}{\includegraphics[height=0.3in]{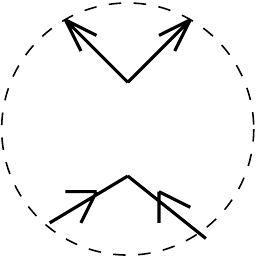}}, while the morphism \raisebox{-8pt}{\includegraphics[height=0.3in]{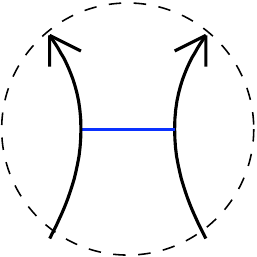}} is the `singular saddle' with domain \raisebox{-8pt}{\includegraphics[height=0.3in]{piecewiseor-diag.pdf}} and range \raisebox{-8pt}{\includegraphics[height=0.3in]{or-diag.pdf}}.
There is a loop in the previous complex, thus applying Lemma~\ref{lemma:delooping}, $[D]$ is isomorphic in $\textit{Kof}_{/h}$ to: 
$$\raisebox{-18pt}{ \includegraphics[height=1.3in]{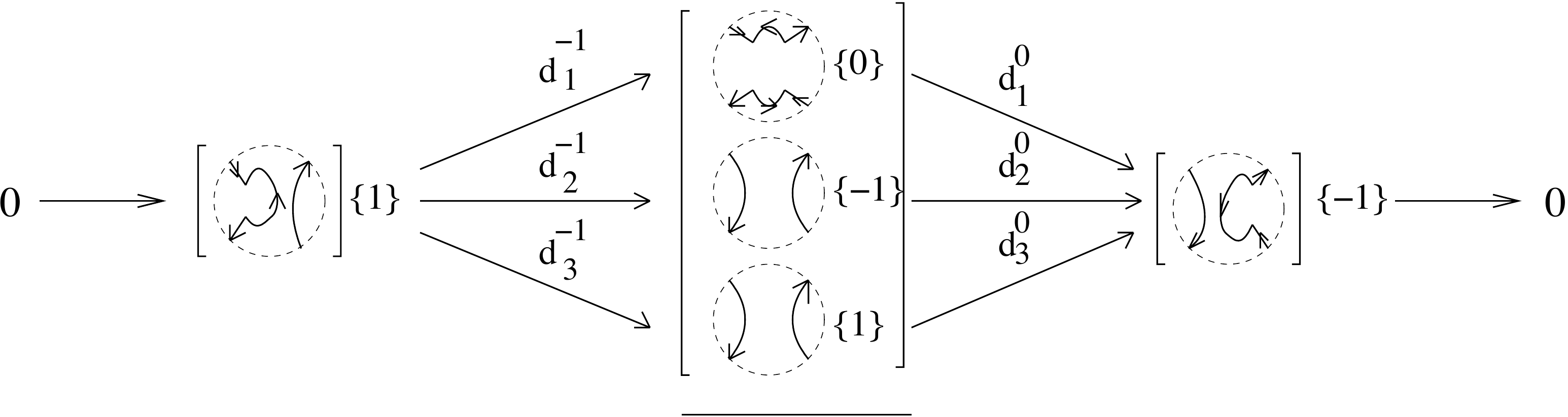}}$$
The above complex is the direct sum of 
$$ \raisebox{-13pt}{\includegraphics[height=.45in]{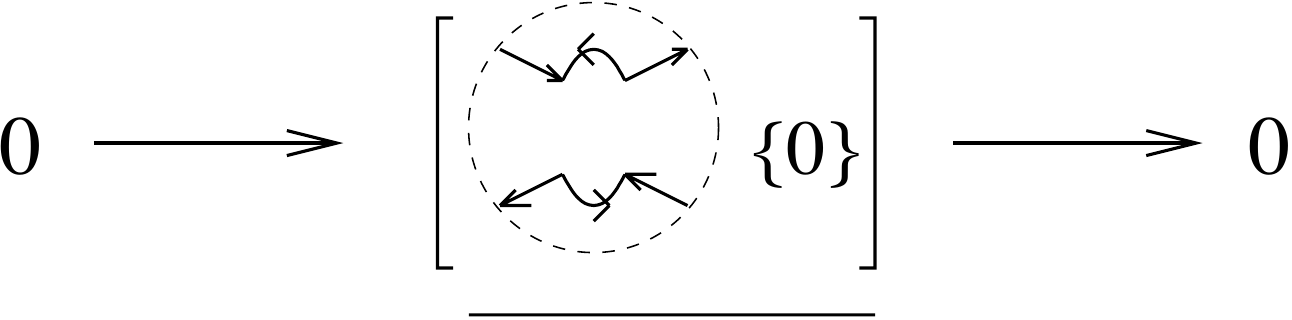}},$$
$$ \raisebox{-13pt}{\includegraphics[height=.53in]{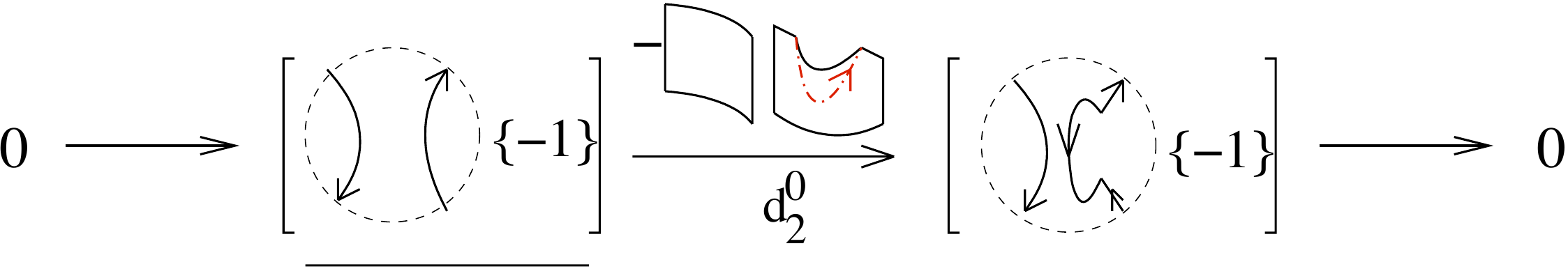}} \,\, \, \, \text{and}$$
$$ \raisebox{-13pt}{\includegraphics[height=.53in]{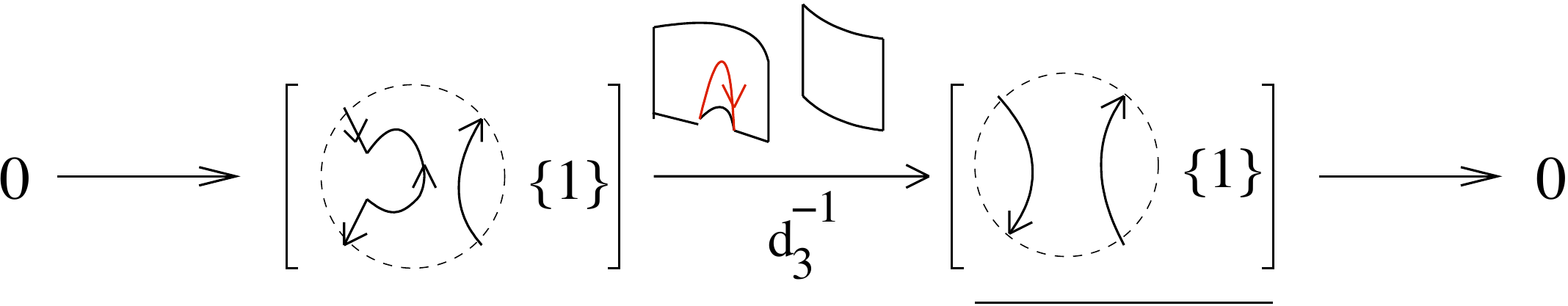}}.$$

The last two complexes are contractible, since there differentials are isomorphisms in $\textit{Foams}_{/\ell},$ and the first complex is isomorphic to $[D']$ (see Lemma~\ref{lemma:isomorphisms}). Removing contractible direct summands we obtain that $[D]$ and $[D']$ are homotopy equivalent.

The other Reidemeister moves have a similar approach.

\subsection*{\textbf{The figure eight knot}}

The figure eight knot diagram in Figure~\ref{figure eight knot} is the connected sum of the two tangle diagrams $T_1 = \raisebox{-8pt}{\includegraphics[height=.35in]{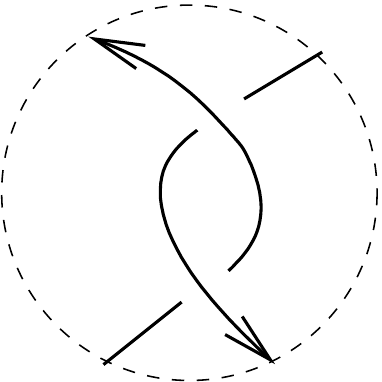}}$ and $T_2 = \raisebox{-8pt}{\includegraphics[height=.35in]{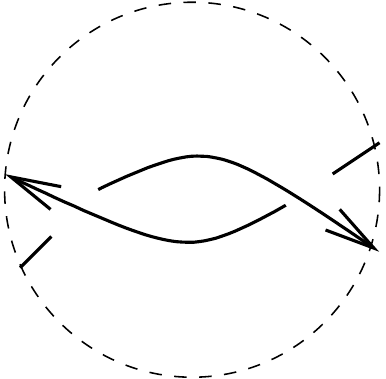}}.$
\begin{figure}[ht]
\centerline{\includegraphics[height=.8in]{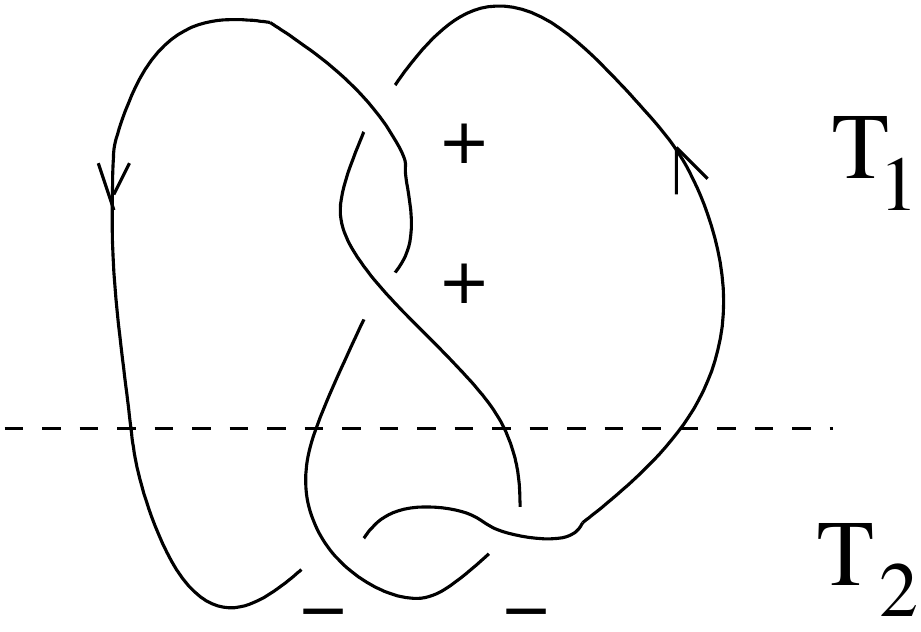}}
\caption{The figure eight knot cut in half}
\label{figure eight knot}
\end{figure}
$$[T_1]:  \left [\raisebox{-8pt}{ \includegraphics[height=.35in]{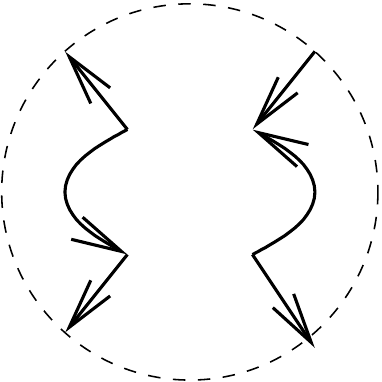}} \right] \{4\} \stackrel{\left (\begin{array}{c} \raisebox{-5pt}{\includegraphics[height=.25in]{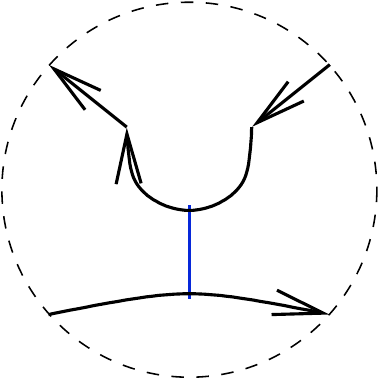}} \vspace{.05in} \\ \raisebox{-5pt}{\includegraphics[height=.25in]{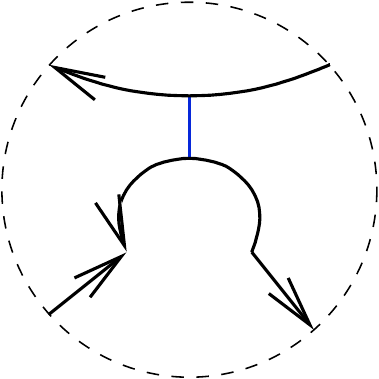}} \end{array} \right )} {\longrightarrow} \left [ \begin{array}{c} \raisebox{-8pt}{\includegraphics[height=.35in]{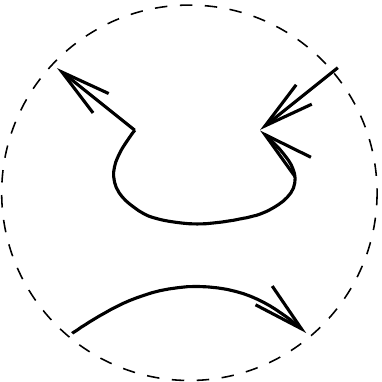}}  \vspace{.05in} \\ \raisebox{-8pt}{\includegraphics[height=.35in]{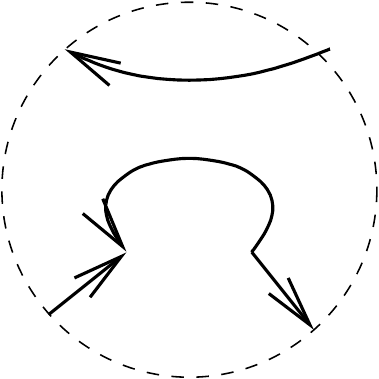}} \end{array} \right ] \{3\} \stackrel{\left ( \begin{array}{cc} \raisebox{-5pt}{\includegraphics[height=.25in]{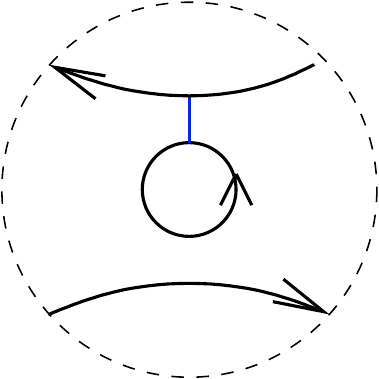}} & -\raisebox{-5pt}{\includegraphics[height=.25in]{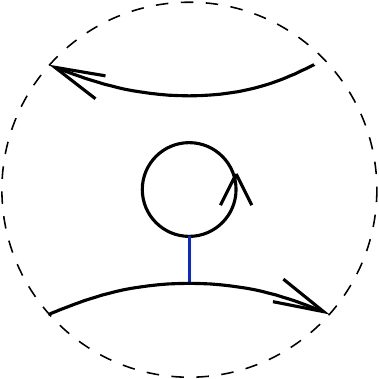}} \end{array}\right)}{\longrightarrow} \underline{\left [ \raisebox{-8pt}{\includegraphics[height=.35in]{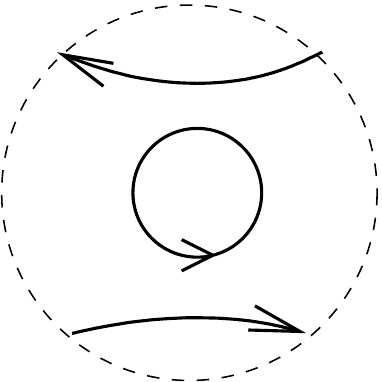}}\right ]\{2\}} 
$$

The object at height $0$ in $[T_1]$ contains a loop. Applying Lemma~\ref{lemma:delooping} and composing the morphisms in the second differential above with those of Figure~\ref{delooping} we get the next complex, which is isomorphic to $[T_1]$:

$$\left [\raisebox{-8pt}{ \includegraphics[height=.35in]{T1-2.pdf}} \right] \{4\} \stackrel{\left (\begin{array}{c} \raisebox{-5pt}{\includegraphics[height=.25in]{T1-d-2-up.pdf}}  \vspace{.05in}\\ \raisebox{-5pt}{\includegraphics[height=.25in]{T1-d-2-down.pdf}} \end{array} \right )} {\longrightarrow} \left [ \begin{array}{c} \raisebox{-8pt}{\includegraphics[height=.35in]{T1-1-up.pdf}}  \vspace{.05in}\\ \raisebox{-8pt}{\includegraphics[height=.35in]{T1-1-down.pdf}} \end{array} \right ] \{3\} \stackrel{\left ( \begin{array}{cc} \raisebox{-5pt}{\includegraphics[height=.25in]{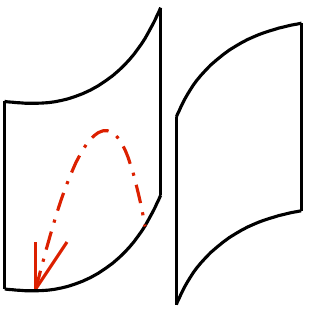}} & -\raisebox{-5pt}{\includegraphics[height=.25in]{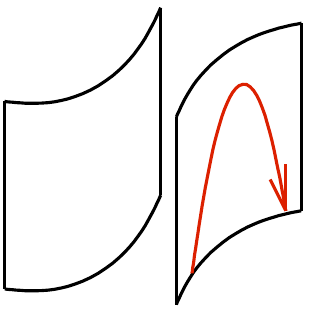}} \vspace{.05in} \\ \raisebox{-5pt}{\includegraphics[height=.25in]{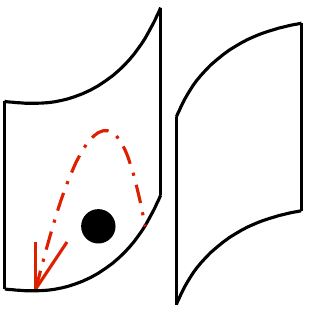}} & -\raisebox{-5pt}{\includegraphics[height=.25in]{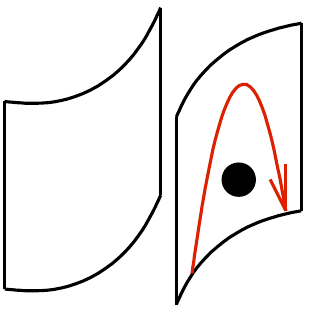}}\end{array}\right)}{\longrightarrow} \underline{\left [ \begin{array}{c}\raisebox{-8pt}{\includegraphics[height=.35in]{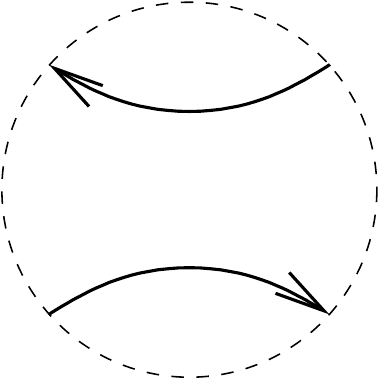}}\{3\}  \vspace{.05in}\\ \raisebox{-8pt}{\includegraphics[height=.35in]{left-right.pdf}}\{1\} \end{array} \right ] }
$$

For the simplicity of drawings, we next apply the isomorphisms~\ref{isomorphisms} to remove pairs of vertices. After this operation, the previous complex is isomorphic to:

$$ \left [\raisebox{-8pt}{ \includegraphics[height=.3in]{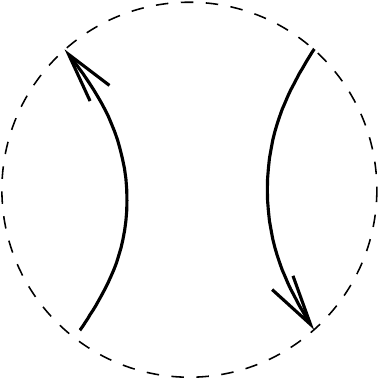}} \right] \{4\} \stackrel{\left (\begin{array}{c} -i \raisebox{-5pt}{\includegraphics[height=.25in]{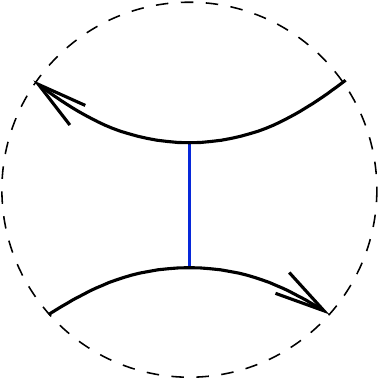}}  \vspace{.05in}\\ -i \raisebox{-5pt}{\includegraphics[height=.25in]{saddle-ud-lr.pdf}} \end{array} \right )} {\longrightarrow} \left [ \begin{array}{c} \raisebox{-8pt}{\includegraphics[height=.3in]{left-right.pdf}}  \vspace{.05in}\\ \raisebox{-8pt}{\includegraphics[height=.3in]{left-right.pdf}} \end{array} \right ] \{3\} \stackrel{\left ( \begin{array}{cc} \,\,\,\,\raisebox{-5pt}{\includegraphics[height=.25in]{left-right.pdf}} & -\raisebox{-5pt}{\includegraphics[height=.25in]{left-right.pdf}}  \vspace{.05in}\\ -\raisebox{-5pt}{\includegraphics[height=.25in]{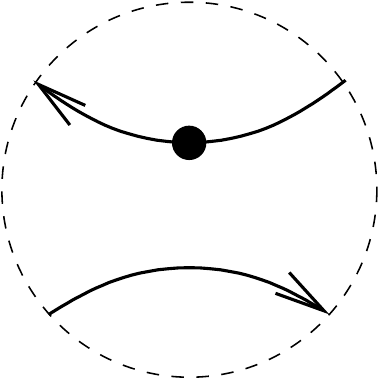}} +h \,\raisebox{-5pt}{\includegraphics[height=.25in]{left-right.pdf}} & \raisebox{-5pt}{\includegraphics[height=.25in]{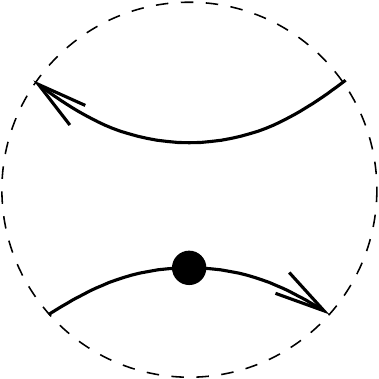}} - h\raisebox{-5pt}{\includegraphics[height=.25in]{left-right.pdf}}\end{array}\right)}{\longrightarrow} \underline{\left [ \begin{array}{c}\raisebox{-8pt}{\includegraphics[height=.3in]{left-right.pdf}}\{3\}  \vspace{.05in}\\ \raisebox{-8pt}{\includegraphics[height=.3in]{left-right.pdf}}\{1\} \end{array} \right ]}$$

When appearing as a cobordism, the symbol \raisebox{-5pt}{\includegraphics[height=.25in]{left-right.pdf}} denotes the identity automorphism of the resolution with the same symbol, that is, it is the union of two `curtains'. Similarly, \raisebox{-5pt}{\includegraphics[height=.25in]{left-right-dup.pdf}} and \raisebox{-5pt}{\includegraphics[height=.25in]{left-right-ddown.pdf}} denote the same cobordism with an extra dot on one of the curtains. Moreover, \raisebox{-5pt}{\includegraphics[height=.25in]{saddle-ud-lr.pdf}} denotes the saddle with domain \raisebox{-5pt}{\includegraphics[height=.25in]{up-down.pdf}} and range \raisebox{-5pt}{\includegraphics[height=.25in]{left-right.pdf}}.

The upper left entry in the second nontrivial differential of the previous complex is an isomorphism, and applying the first part of Lemma~\ref{lemma:Gaussian elimination} we arrive at the complex below:
$$  \left [\raisebox{-8pt}{ \includegraphics[height=.35in]{up-down.pdf}} \right] \{4\} \stackrel{\left (\begin{array}{c} 0  \vspace{.05in}\\ -i \raisebox{-5pt}{\includegraphics[height=.25in]{saddle-ud-lr.pdf}} \end{array} \right )} {\longrightarrow} \left [ \begin{array}{c} \raisebox{-8pt}{\includegraphics[height=.35in]{left-right.pdf}}  \vspace{.05in}\\ \raisebox{-8pt}{\includegraphics[height=.35in]{left-right.pdf}} \end{array} \right ] \{3\} \stackrel{\left ( \begin{array}{cc} \raisebox{-5pt}{\includegraphics[height=.25in]{left-right.pdf}} & 0  \vspace{.05in}\\ 0 & \raisebox{-5pt}{\includegraphics[height=.25in]{left-right-ddown.pdf}} - \raisebox{-5pt}{\includegraphics[height=.25in]{left-right-dup.pdf}}\end{array}\right)}{\longrightarrow} \underline{\left [ \begin{array}{c}\raisebox{-8pt}{\includegraphics[height=.35in]{left-right.pdf}}\{3\}  \vspace{.05in}\\ \raisebox{-8pt}{\includegraphics[height=.35in]{left-right.pdf}}\{1\} \end{array} \right ]}$$

Removing the contractible summand
$$ 0 \longrightarrow \left [ \raisebox{-8pt}{\includegraphics[height=.35in]{left-right.pdf}}\right ]\{3\} \stackrel{ \left (\raisebox{-5pt}{\includegraphics[height=.25in]{left-right.pdf}} \right)} {\longrightarrow} \underline{\left [\raisebox{-8pt}{\includegraphics[height=.35in]{left-right.pdf}} \right ]\{3\}}\longrightarrow 0,$$
we obtain the complex $\mathcal{C}_1$ given below, which is isomorphic to $[T_1]$ in $\textit{Kof}_{/h}.$
$$\mathcal{C}_1: \quad   \left [\raisebox{-8pt}{ \includegraphics[height=.35in]{up-down.pdf}} \right] \{4\} \stackrel{ \left (-i \,\raisebox{-5pt}{\includegraphics[height=.25in]{saddle-ud-lr.pdf}} \right)} {\longrightarrow} \left [ \raisebox{-8pt}{\includegraphics[height=.35in]{left-right.pdf}}\right ]\{3\} \stackrel{ \left ( \,\raisebox{-5pt}{\includegraphics[height=.25in]{left-right-ddown.pdf}} - \,\raisebox{-5pt}{\includegraphics[height=.25in]{left-right-dup.pdf}}\right)} {\longrightarrow} \underline{\left[ \raisebox{-8pt}{\includegraphics[height=.35in]{left-right.pdf}}\right ]\{1\}}$$
The complex $[T_2]$ associated to the other half of the figure eight knot diagram is computed and simplified similarly. It turns out that it is isomorphic in $\textit{Kof}_{/h}$ to the complex $\mathcal{C}_2$ depicted below.
$$\mathcal{C}_2: \quad  \underline{\left [\raisebox{-8pt}{ \includegraphics[height=.35in]{up-down.pdf}} \right] \{-1\}} \stackrel{ \left (i \,\raisebox{-5pt}{\includegraphics[height=.25in]{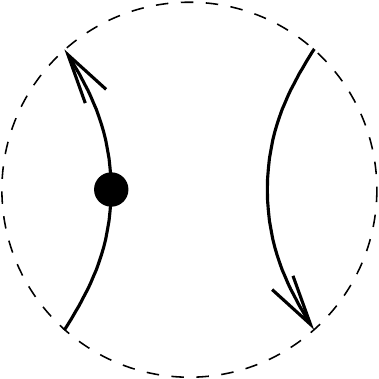}} -i \,\raisebox{-5pt}{\includegraphics[height=.25in]{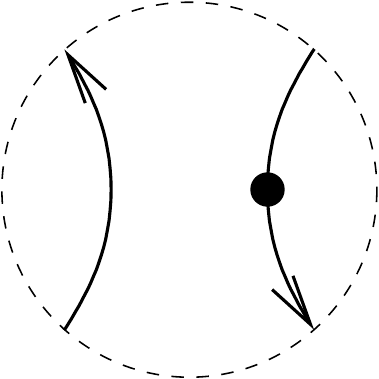}} \right)} {\longrightarrow} \left [ \raisebox{-8pt}{\includegraphics[height=.35in]{up-down.pdf}}\right ]\{-3\} \stackrel{ \left ( -\,\raisebox{-5pt}{\includegraphics[height=.25in]{saddle-ud-lr.pdf}} \right)} {\longrightarrow} \left[ \raisebox{-8pt}{\includegraphics[height=.35in]{left-right.pdf}}\right ]\{-4\}$$
Next step is to take the `formal tensor product' of $\mathcal{C}_1$ with $\mathcal{C}_2$ using the same side-by-side composition that one has to use to get from $T_1$ and $T_2$ the figure eight knot diagram. As a result, we obtain the double complex $\mathcal{C}$ shown below, in which we smoothed out the resolutions and cobordisms; we also canceled the four morphisms obtained on the upper right of the diagram, as they are differences of the same cobordism.
$$ \includegraphics[height=3.5in]{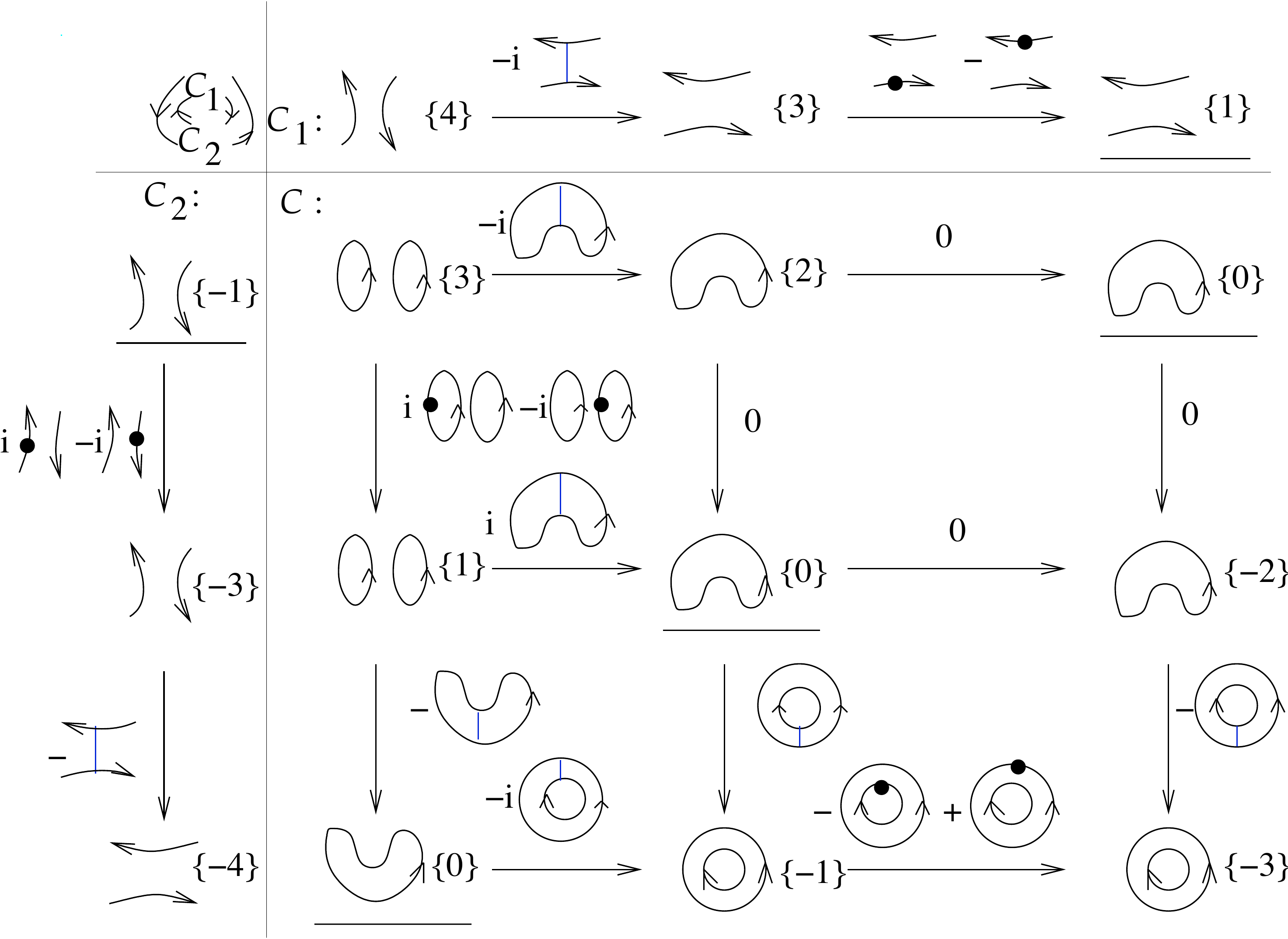}$$

The following step is to replace every loop with a pair of degree-shifted empty sets as in Lemma~\ref{lemma:delooping}, and to replace the differentials with their compositions with the isomorphisms of Figure~\ref{delooping}. As every object of $\mathcal{C}$ contains only loops, we arrive at the complex $\Lambda_1$ depicted below, in which all objects are degree-shifted empty sets and all morphisms are matrices of scalar multiples of the empty cobordism---recall that we are working modulo the local relations $\ell$ and all closed foams reduce to an element of the ground ring $\mathbb{Z}[i, a, h].$

Let us recall from~\cite{CC2} the structure maps for the Frobenius algebra defined on $\mathcal{A} = \mathbb{Z}[i, a, h]/(X^2 -h X -a) = \brak{1, X}_{\mathbb{Z}[i, a, h]}.$ The unit map $\iota \co \mathbb{Z}[i][a, h] \rightarrow \mathcal{A}$ and counit map $ \epsilon \co \mathcal{A} \rightarrow \mathbb{Z}[i][a, h]$ are given by $\iota(1) = 1$ and $ \epsilon(1) = 0, \, \epsilon(X) = 1,$ respectively.
Multiplication $m \co \mathcal{A} \otimes \mathcal{A} \rightarrow \mathcal{A}$ and comultiplication $\Delta \co \mathcal{A} \rightarrow \mathcal{A} \otimes \mathcal{A}$ are defined by
$$ \begin{cases}
 m(1 \otimes X) =X, & m(X \otimes 1) = X\\ 
m(1 \otimes 1) =1, & m(X \otimes X) =hX+ a
 \end{cases},\quad
 \begin{cases}
 \Delta(1) = 1 \otimes X + X \otimes 1-h1\otimes 1\\ 
 \Delta(X) = X \otimes X + a 1 \otimes 1.\end{cases} $$
We  will use the basis $(1,X)$ of the algebra $\mathcal{A}$ and write the maps $m$ and $\Delta$ relative to this basis. The cobordism \raisebox{-5pt}{\includegraphics[height=.25in]{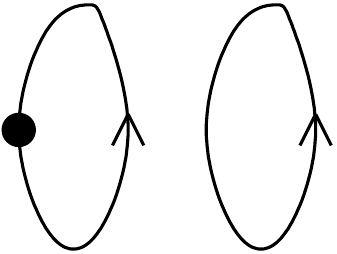}} is the multiplication by $X$ endomorphism of $\mathcal{A}$  on the first component of $\mathcal{A} \otimes \mathcal{A}$ and the identity map on the second component. Likewise, \raisebox{-5pt}{\includegraphics[height=.25in]{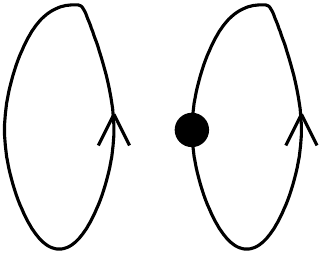}} is the identity on the first component and multiplication by $X$ endomorphism of $\mathcal{A}$ on the second component of the tensor product. Therefore, these cobordisms are defined by the following rules:
\begin{align} \raisebox{-5pt}{\includegraphics[height=.25in]{two-circles-leftd.pdf}} &= \left \{ \begin {array}{ccccc} 1 \otimes 1 & \rightarrow &X \otimes 1 \\
1 \otimes X &\rightarrow & X \otimes X \\
X \otimes 1 & \rightarrow & X ^2\otimes 1 &= &h (X\otimes 1) + a (1\otimes 1)\\
X \otimes X & \rightarrow & X ^2\otimes X &= & h(X \otimes X) + a (1 \otimes X) \end{array} \right. \\ 
\raisebox{-5pt}{\includegraphics[height=.25in]{two-circles-rightd}} &= \left \{ \begin {array}{ccccc} 1 \otimes 1 &\rightarrow &1 \otimes X \\
1 \otimes X &\rightarrow & 1 \otimes X^2 & =&h (1 \otimes X)+ a (1\otimes 1) \\
X \otimes 1 &\rightarrow & X \otimes X \\
X \otimes X &\rightarrow& X \otimes X^2 &=&h (X\otimes X) + a (X \otimes 1) \end{array} \right.
\end{align}  
The matrix of the cobordism $i \raisebox{-5pt}{\includegraphics[height=.25in]{two-circles-leftd.pdf}} -i \raisebox{-5pt}{\includegraphics[height=.25in]{two-circles-rightd.pdf}}$ relative to the basis $(1\otimes 1, 1\otimes X, X \otimes 1, X \otimes X )$ of the tensor product $\mathcal{A} \otimes \mathcal{A}$ is:
$$ i \, \raisebox{-5pt}{\includegraphics[height=.3in]{two-circles-leftd.pdf}} -i \,\raisebox{-5pt}{\includegraphics[height=.25in]{two-circles-rightd.pdf}} = \ \  \left ( \begin{array}{cccc} 0 & -ai & ai & 0 \\ -i & -hi & 0 & ai \\ i & 0 & hi & -ai \\ 0 & i & -i & 0 \end{array} \right ).$$
Likewise we have
$ \,\, \raisebox{-5pt}{\includegraphics[height=.3in]{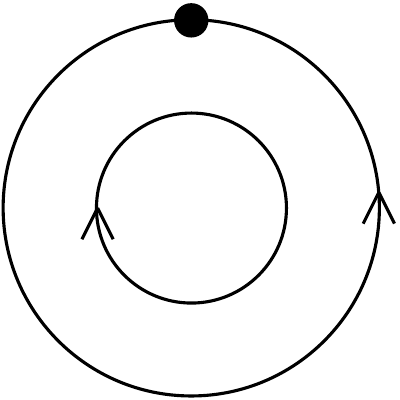}} - \raisebox{-5pt}{\includegraphics[height=.25in]{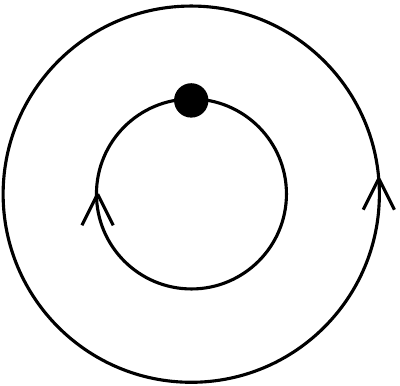}} = \ \  \left ( \begin{array}{cccc} 0 & -a & a & 0 \\ -1 & -h & 0 & a \\ 1 & 0 & h & -a \\ 0 & 1 & -1 &0 \end{array} \right ).$ 

We are now ready to write the complex $\Lambda_1.$

$$\Lambda_1:\qquad \raisebox{-270pt}{\includegraphics[height=4in]{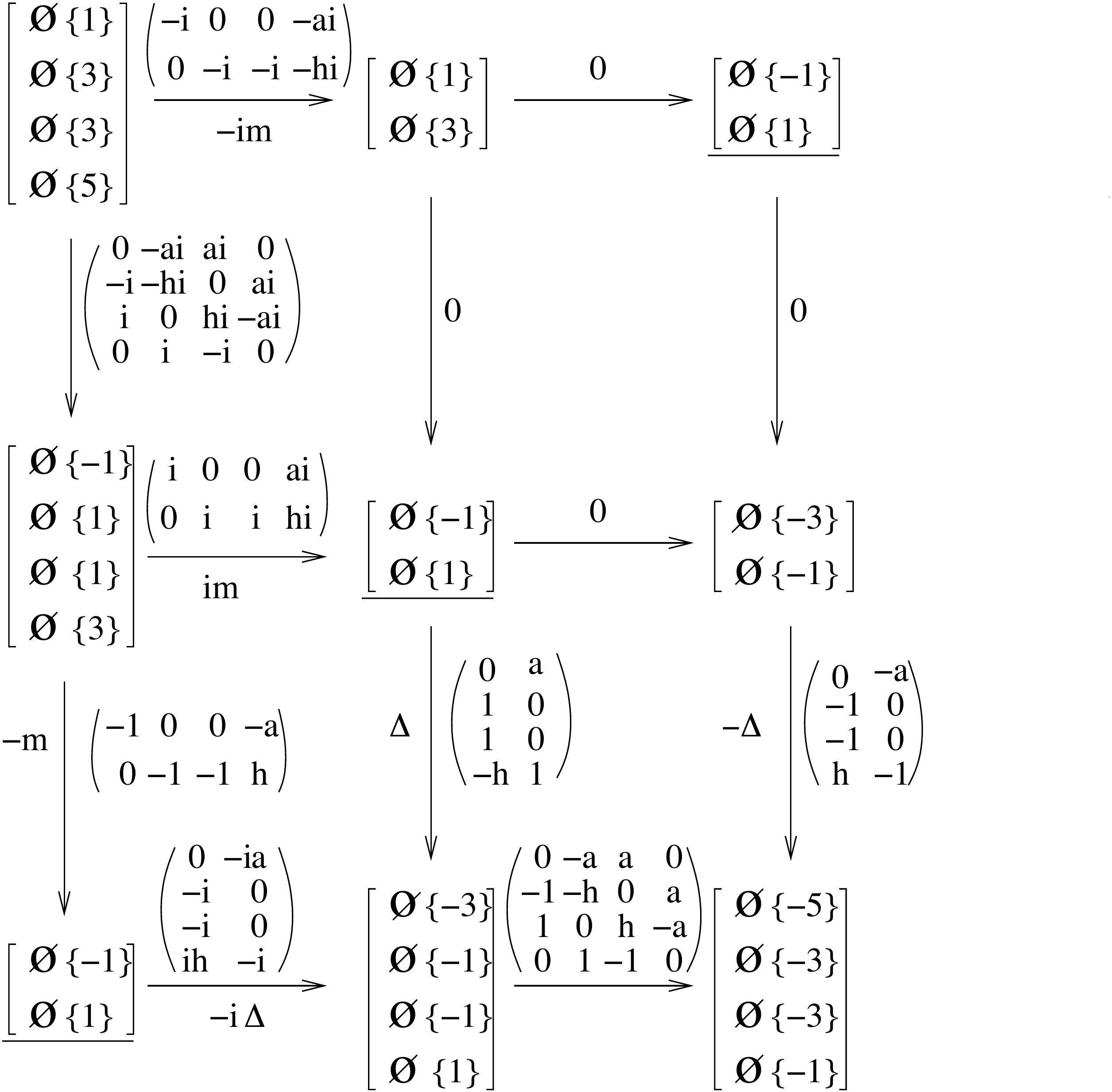}}$$

There are many isomorphisms in $\Lambda_1,$ thus we repeatedly apply Lemma~\ref{lemma:Gaussian elimination} until no invertible entries remain in any of the matrices. Adding relations $a = 0 = h$ and working over $\mathbb{C}$ (note that we can also use here $\mathbb{Q}(i)),$ any non-zero number is invertible. Henceforth we obtain the double complex $\Lambda_2$ given below, in which all matrices are $0.$

$$\Lambda_2:\qquad \raisebox{-120pt}{\includegraphics[height=1.8in]{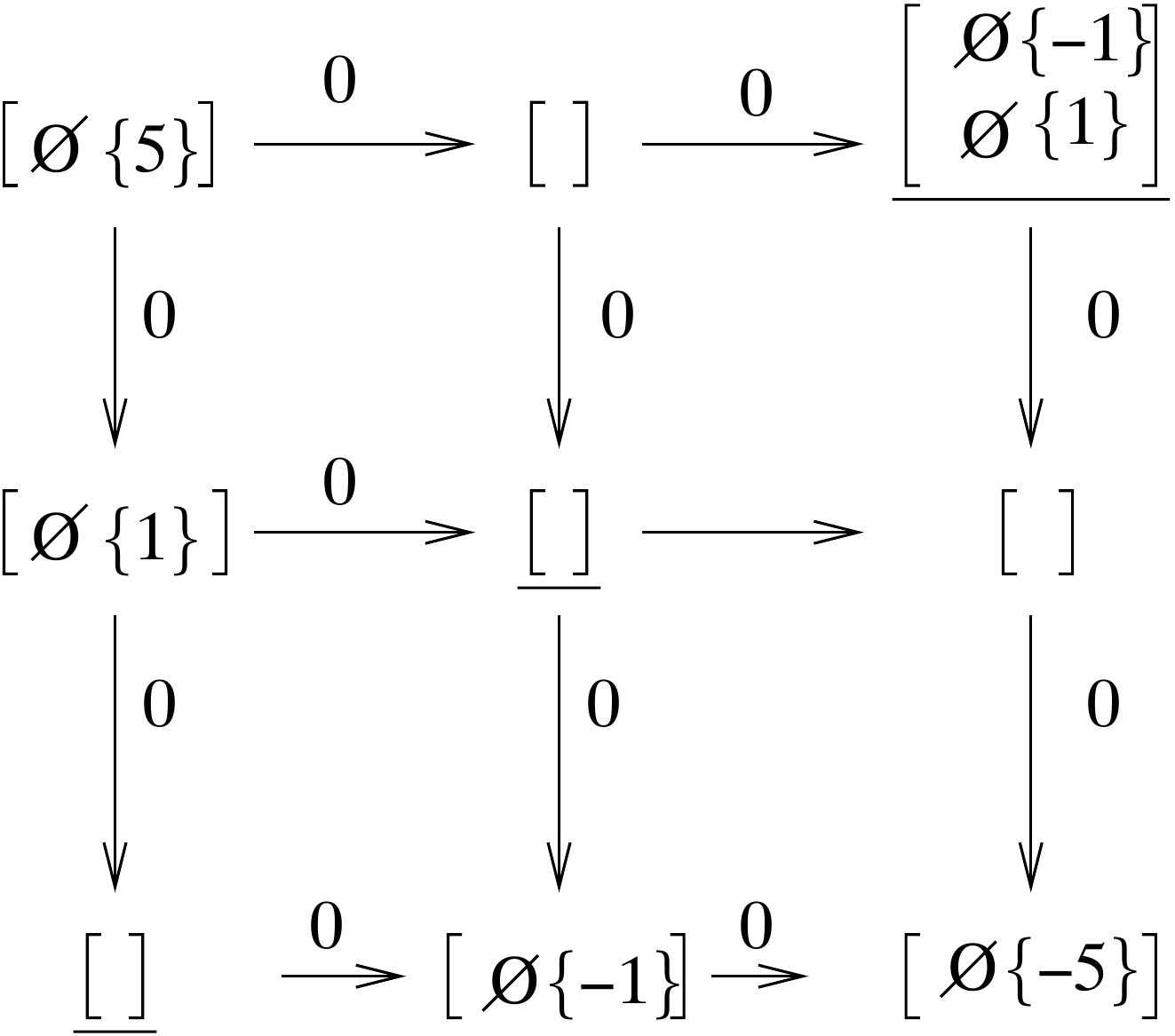}}$$

Taking the total complex of $\Lambda_2$ we arrive at the complex $\Lambda_3$ which is homotopy equivalent to $[K],$ where $K$ is the figure eight knot diagram we started with:
$$\Lambda_3: \raisebox{-10pt}{ \includegraphics[height=.4in]{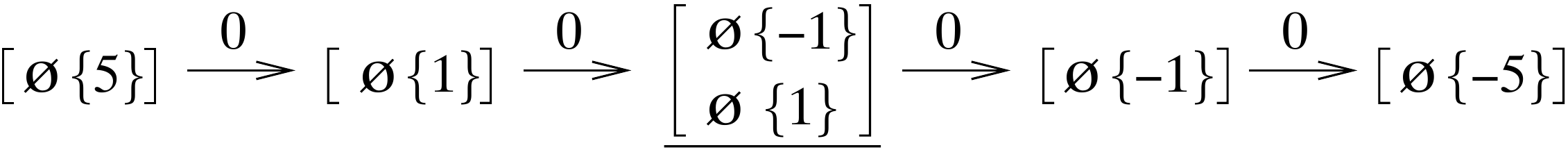}}.$$

We need now to apply the functor $\mathcal{F}_{\emptyset}$ defined in~\cite{CC2}, to obtain an ordinary complex with objects graded vector spaces over $\mathbb{C}$ and take its cohomology. Since $\mathcal{F}_{\emptyset}(\emptyset) = \mathbb{C} \{0\}$ we have:
$$\mathcal{F}_{\emptyset}(\Lambda_3): \mathbb{C}\{5\} \stackrel{0}{\longrightarrow} \mathbb{C}\{1\} \stackrel{0}{\longrightarrow} \underline{ \mathbb{C}\{-1\} \oplus \mathbb{C}\{1\}} \stackrel{0}{\longrightarrow}\mathbb{C}\{-1\} \stackrel{0}{\longrightarrow} \mathbb{C}\{-5\}.$$

Computing the cohomology of the complex $\mathcal{F}_{\emptyset}(\Lambda_3)$ we obtain that the cohomology group over $\mathbb{C}$ of the figure eight knot  is 6--dimensional, with generators in bidegrees $(-2,5), (-1,1), (0, -1), (0,1)$, $ (1, -1)$ and $(2,-5).$ That is, after adding the relations $a = 0 = h$, the cohomology groups of the figure eight knot are:
$$\mathcal{H}^{i,j}(K) \otimes_{\mathbb{Z}[i]} \mathbb{C}= \begin{cases}\mathbb{C} \ \ \ \ \text{for}\ \ \  (i,j) \in \{(-2,5), (-1,1), (0, -1), (0,1), (1, -1),(2,-5)\} \\ 0 \ \ \ \ \text{otherwise}. \end{cases}$$ 

\begin{remark}
When implementing the delooping and Gaussian elimination tools, there is a better way than computing the `half knot' invariants and putting the results together, as noted in~\cite[Section 7]{BN2}. Instead, one picks a certain crossing of the knot and adds one crossing at a time. After each crossing is added, the two tools are used to simplify the result before moving to the next crossing.
\end{remark}

\section{Universal $\mf{sl}(2)$ foam cohomology without dots}\label{sec:no dots}

In this section we describe a purely topological variant of our universal $\mf{sl}(2)$ dotted foam theory, in the sense that no dots are present on singular cobordisms. However, one has to pay a price for it, because this theory works if $2^{-1}$ exists in the ground ring. Moreover, we will also lose the beautiful geometric interpretation of certain algebra structure properties on $\mathcal{A}$. For example, relation (2D) is the geometric counterpart of the identity $X^2  = hX + a$ in $\mathcal{A},$ and the surgery formula (SF) corresponds to $\Delta(1) = 1 \otimes X + X \otimes 1 -h\, 1 \otimes 1;$ both of them are lost.     

The main idea is to use the \textit{genus reduction} formula from~\cite{CC2}, which for our purpose here, we write it in the following form:
\[\raisebox{-8pt}{\includegraphics[height=.3in]{capordot.pdf}} = \frac{1}{2} (\, \raisebox{-8pt}{\includegraphics[height=.3in]{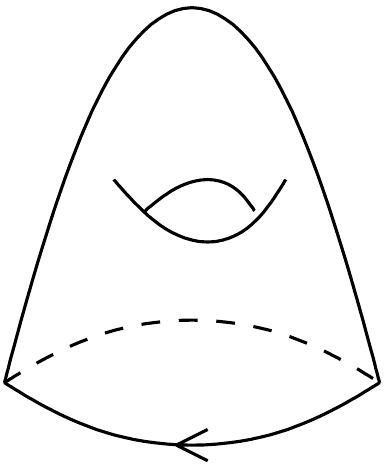}} + h\, \raisebox{-8pt}{\includegraphics[height=.3in]{capor.pdf}}\,  ).\]
We consider the ground ring $R: = \mathbb{Z}[\frac{1}{2}, i, a, h]$ and the $R$-module $\mathcal{A}' = R[X]/(X^2 - hX -a)$ with generators $1$ and $X,$ on which we consider the same Frobenius structure maps we used for $\mathcal{A}.$

The local relations $\ell$ = (2D, SF, S, UFO) from~\cite{CC2} are replaced by $\tilde{\ell} = (G2, \tilde{SF}, \tilde{S}, T,\tilde{UFO})$ given below. Note that the \textit{genus two reduction} formula $(G2)$ is the replacement of the \textit{two dots} relation (2D), and that the \textit{torus identity} $(T)$ corresponds to the second identity in (S).

$$\xymatrix@R=2mm{
(\tilde{SF}) \hspace{1cm} \raisebox{-22pt}{\includegraphics[height=.7in]{surgery.pdf}} = \displaystyle\frac{1}{2} \raisebox{-22pt}{\includegraphics[height=.7in]{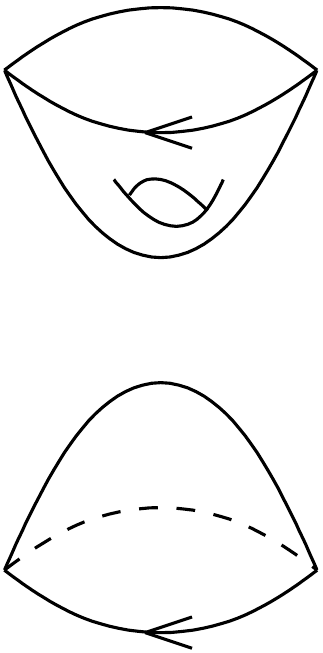}} + \frac{1}{2}\raisebox{-22pt}{\includegraphics[height=.7in]{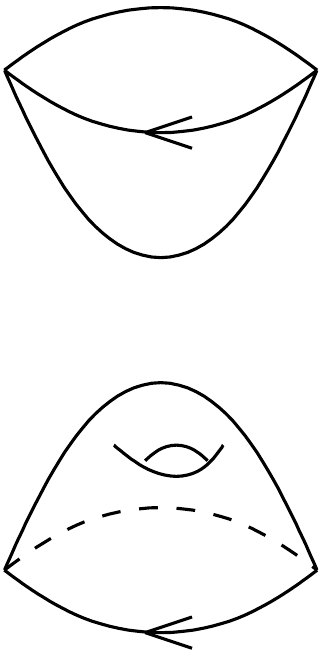}} \hspace{2cm}  \raisebox{-10pt}{\includegraphics[width=0.4in]{sph.pdf}}=0 &(\tilde{S}) \\
\raisebox{-13pt}{\includegraphics[height=.4in]{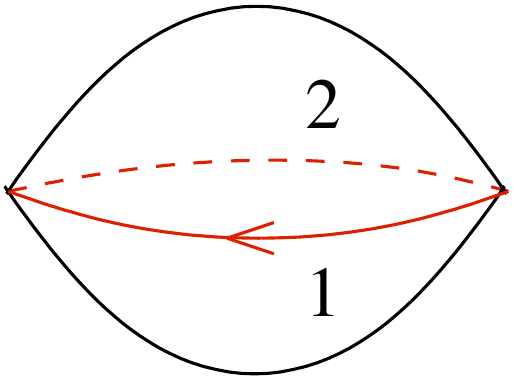}} = 0 =  \raisebox{-13pt}{\includegraphics[height=.4in]{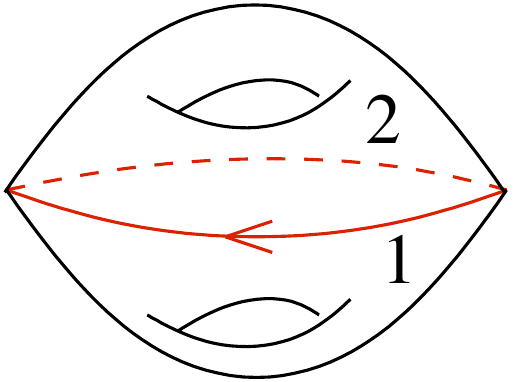}}\,, \quad \raisebox{-13pt}{\includegraphics[height=.4in]{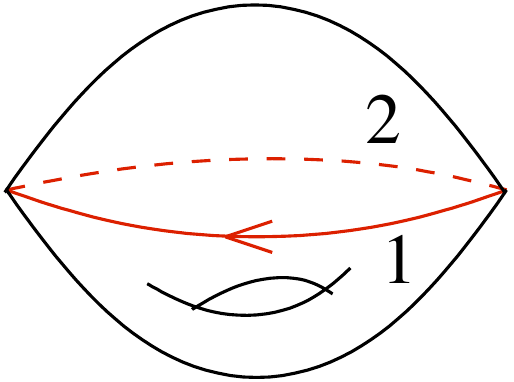}} = 2i\,, \quad \raisebox{-13pt}{\includegraphics[height=.4in]{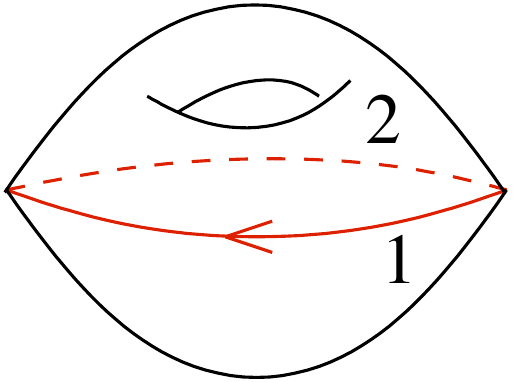}} = -2i & (\tilde{UFO})\\
\ (G2)\ \ \ \ \ \   
\raisebox{-8pt}{\includegraphics[height=.35in]{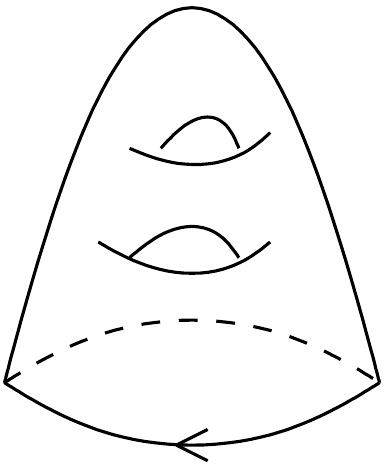}} = (h^2 + 4a)\, \raisebox{-8pt}{\includegraphics[height=.35in]{capor.pdf}} \hspace{2.3cm} \raisebox{-5pt}{\includegraphics[width=0.6in]{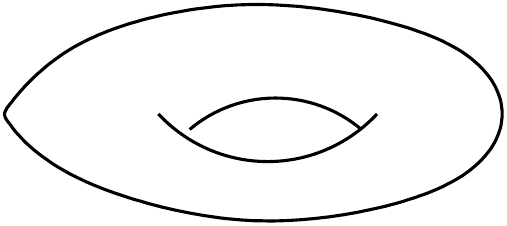}}=2 & (T)
}$$

The \textit{curtain identities} still hold, as well as the isomorphisms given in Lemma~\ref{lemma:isomorphisms}. With the new relations $\tilde{\ell},$ the \textit{cut-neck relation} (CN) and identity (RSC) become $(\tilde{CN})$ and $(\tilde{RSC})$ respectively, which are depicted below.
$$\xymatrix@R=2mm{
\raisebox{-22pt}{\includegraphics[width=0.3in, height=.8in]{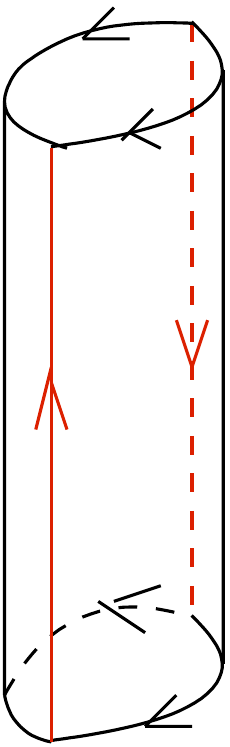}}= \displaystyle -\frac{i}{2}
\raisebox{-25pt}{\includegraphics[width=0.35in, height=0.85in]{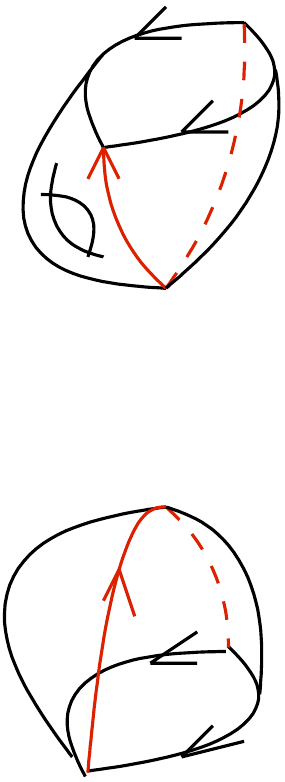}} - \frac{i}{2}
\raisebox{-25pt}{\includegraphics[width=0.35in, height=0.85in]{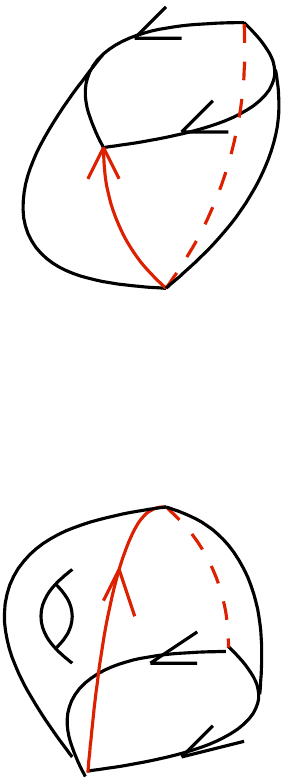}}
& (\tilde{CN}) \\
 \raisebox{-22pt}{\includegraphics[height=0.7in]{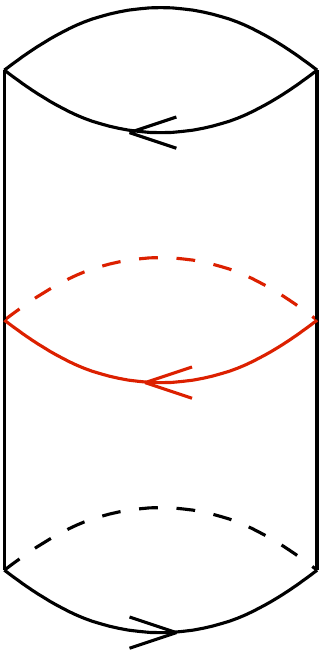}} = \displaystyle\frac{i}{2}
\raisebox{-22pt}{\includegraphics[height=0.7in]{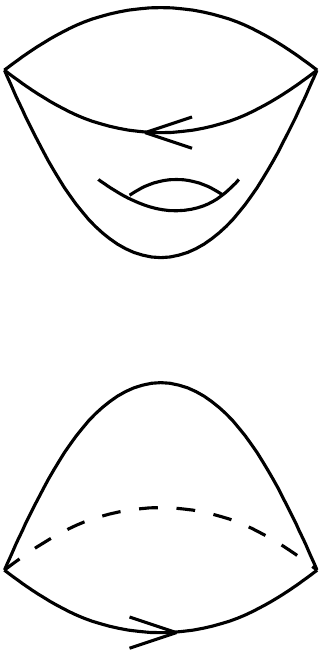}} - \frac{i}{2}
\raisebox{-22pt}{\includegraphics[height=0.7in]{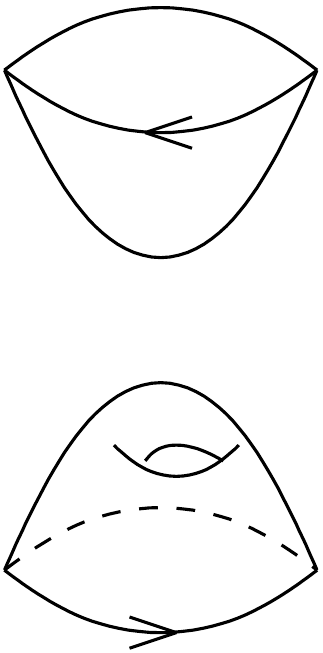}} & (\tilde{RSC}) 
}$$

Interestingly, relations~\ref{exchanging dots} are replaced here by:
\begin{align}
\raisebox{-8pt}{\includegraphics[height=.35in]{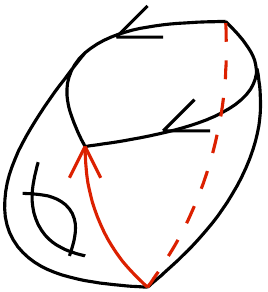}} + \raisebox{-8pt}{\includegraphics[height=.35in]{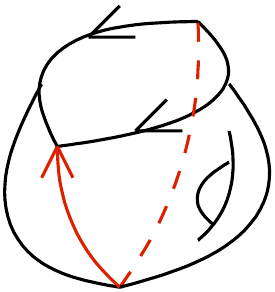}} = 0 \hspace{2cm} \raisebox{-8pt}{\includegraphics[height=.35in]{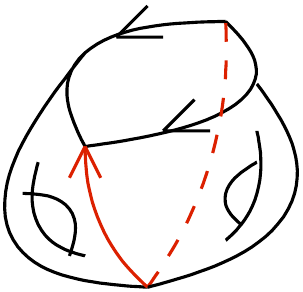}} = -(h^2 + 4a)\, \raisebox{-8pt}{\includegraphics[height=.35in]{cupsa.pdf}}
\end{align}
Therefore, handles can be exchanged between two neighboring facets of a foams, at the expense of a minus sign. 

We denote by $\textit{Foams}_{/\tilde{\ell}}$ the category of dotless foams modulo the new local relations $\tilde{\ell}.$ The tautological functor $\textit{F} \co \textit{Foams}_{/\tilde{\ell}}(\emptyset) \longrightarrow R$-Mod is defined on links as before, namely on objects is defined by $\mathcal{F}(\Gamma): = \Hom_{\textit{Foams}_{/\tilde{\ell}}(\emptyset)}(\emptyset, \Gamma)$ and on morphisms by composition on the left. Using the isomorphisms from Lemma~\ref{lemma:isomorphisms}, any resolution of a link diagram is isomorphic in $\textit{Foams}_{/\tilde{\ell}}$ with a disjoint union of oriented loops and closed webs with two vertices. Note that also a closed web with two vertices is isomorphic in $\textit{Foams}_{/\tilde{\ell}}$ to an oriented loop. Repeatedly applying the local relations $\tilde{\ell}$ we get to foams in which every component has one boundary curve (which is either \raisebox{-5pt}{\includegraphics[height=.18in]{unknot-clockwise.pdf}} or \raisebox{-3pt}{\includegraphics[height=.16in]{circle2sv.pdf}}) and is either of genus $0$ or of genus $1.$ In particular, $\mathcal{F}(\raisebox{-5pt}{\includegraphics[height=.18in]{unknot-clockwise.pdf}})$ is an $R$-module generated by  $v_{-} = \raisebox{-5pt}{\includegraphics[height=.2in]{cupor.pdf}}$ and $v_{+} = \frac{1}{2} \,\raisebox{-5pt}{\includegraphics[height=.2in]{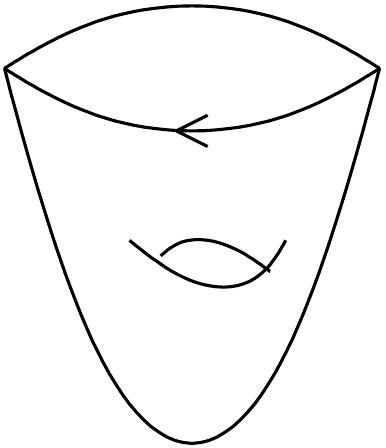}} + \frac{1}{2} h\, \raisebox{-5pt}{\includegraphics[height=.2in]{cupor.pdf}}$ and $\mathcal{F}(\raisebox{-3pt}{\includegraphics[height=.16in]{circle2sv.pdf}})$ is an $R$-module generated by  $v'_{-} = \raisebox{-5pt}{\includegraphics[height=.2in]{cupsa.pdf}}$ and $v'_{+} = \frac{1}{2}\, \raisebox{-5pt}{\includegraphics[height=.2in]{cupsa-g1.pdf}} + \frac{1}{2} h \,\raisebox{-5pt}{\includegraphics[height=.2in]{cupsa.pdf}}.$

It is easy to see that there are degree-preserving $R$-module isomorphisms $\mathcal{F}(\raisebox{-5pt}{\includegraphics[height=.18in]{unknot-clockwise.pdf}}) \cong \mathcal{A} \cong \mathcal{F}(\raisebox{-3pt}{\includegraphics[height=.16in]{circle2sv.pdf}})$ that map $v_{-}, v'_{-}$ to $1$ and $v_{+}, v'_{+}$ to $X.$ Moreover, one can verify that with this basis for $\mathcal{F}(\raisebox{-5pt}{\includegraphics[height=.18in]{unknot-clockwise.pdf}})$ the tautological functor $\mathcal{F}$ behaves---on links---in the same manner as the TQFT functor $\mathsf{F} \otimes \mathbb{Z}[\frac{1}{2}]$ defined by the Frobenius system corresponding to $\mathcal{A}',$  where $\mathsf{F}$ is the TQFT functor used in~\cite{CC2}. Furthermore, $\mathsf{F} \otimes \mathbb{Z}[\frac{1}{2}]$ satisfies relations $\tilde{\ell}.$ 

For example, to show that $\mathsf{F} \otimes \mathbb{Z}[\frac{1}{2}]$ satisfies relation $(G2)$ we need to check $\epsilon \circ m \circ \Delta \circ m \circ \Delta =  (h^2 + 4a) \epsilon$ (recall that we read cobordisms from bottom to top). For $(\tilde{SF})$ we need to verify that $\id = \frac{1}{2} (m \circ \Delta \circ \iota \circ \epsilon + \iota \circ \epsilon \circ m \circ \Delta).$ Both equalities hold. Moreover, $\mathsf{F} \otimes \mathbb{Z}[\frac{1}{2}]$ satisfies relations $(\tilde{S})$ and $(T)$ since $\epsilon \circ \iota = 0$ and $\epsilon \circ m \circ \Delta \circ \iota = 2.$ Finally, a singular circle corresponds to multiplication by $\pm i$ endomorphism of $\mathcal{A}'$ (see~\cite{CC2}), and we let the details of checking that $\mathsf{F} \otimes \mathbb{Z}[\frac{1}{2}]$ satisfies relations $(\tilde{UFO})$ to the reader.   

On the other hand we have:
\[\raisebox{-8pt}{\includegraphics[height=.3in]{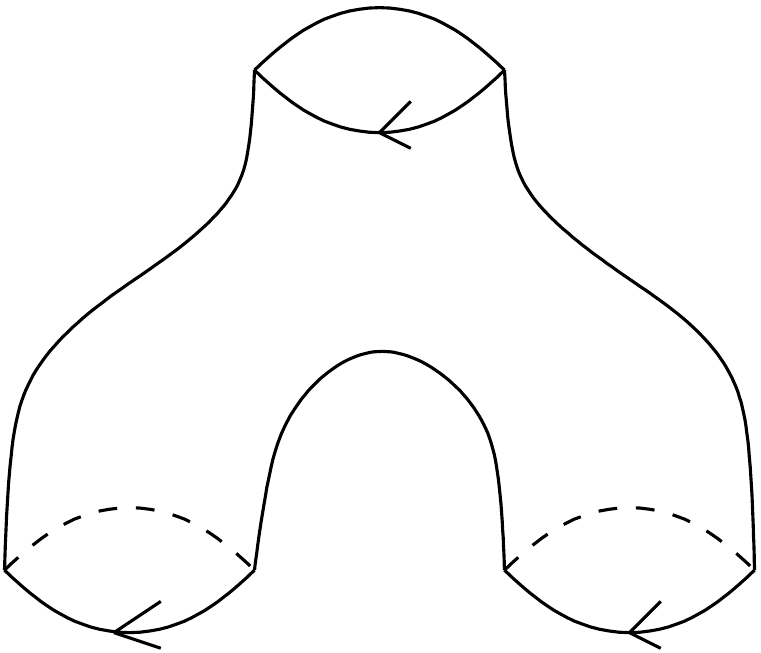}} \circ (\,\raisebox{-5pt}{\includegraphics[height=.2in]{cupor.pdf}} \otimes \raisebox{-5pt}{\includegraphics[height=.2in]{cupor.pdf}}\,) = \raisebox{-5pt}{\includegraphics[height=.2in]{cupor.pdf}} \]
\[\raisebox{-8pt}{\includegraphics[height=.3in]{pair-of-pants.pdf}} \circ [\,\raisebox{-5pt}{\includegraphics[height=.2in]{cupor.pdf}} \otimes (\,\frac{1}{2}\,\raisebox{-5pt}{\includegraphics[height=.2in]{cupor-g1.pdf}} + \frac{h}{2}\, \raisebox{-5pt}{\includegraphics[height=.2in]{cupor.pdf}}  \,)]  = \frac{1}{2}\,\raisebox{-5pt}{\includegraphics[height=.2in]{cupor-g1.pdf}} + \frac{h}{2} \,\raisebox{-5pt}{\includegraphics[height=.2in]{cupor.pdf}} =  \raisebox{-8pt}{\includegraphics[height=.3in]{pair-of-pants.pdf}} \circ [ (\,\frac{1}{2}\,\raisebox{-5pt}{\includegraphics[height=.2in]{cupor-g1.pdf}} + \frac{h}{2}\, \raisebox{-5pt}{\includegraphics[height=.2in]{cupor.pdf}}\,) \otimes \raisebox{-5pt}{\includegraphics[height=.2in]{cupor.pdf}}  \,]  \]

\[\raisebox{-8pt}{\includegraphics[height=.3in]{pair-of-pants.pdf}} \circ [ (\,\frac{1}{2}\,\raisebox{-5pt}{\includegraphics[height=.2in]{cupor-g1.pdf}} + \frac{h}{2}\, \raisebox{-5pt}{\includegraphics[height=.2in]{cupor.pdf}}\,) \otimes (\,\frac{1}{2}\,\raisebox{-5pt}{\includegraphics[height=.2in]{cupor-g1.pdf}} + \frac{h}{2}\, \raisebox{-5pt}{\includegraphics[height=.2in]{cupor.pdf}}\,)] =h\,(\, \frac{1}{2}\,\raisebox{-5pt}{\includegraphics[height=.2in]{cupor-g1.pdf}} + \frac{h}{2}\, \raisebox{-5pt}{\includegraphics[height=.2in]{cupor.pdf}}\,) + a\, \raisebox{-5pt}{\includegraphics[height=.2in]{cupor.pdf}}  \]

which are equivalent to $m(1 \otimes 1) =1, m(1 \otimes X) = X = m(X \otimes 1), m(X \otimes X) =hX+ a.$

Moreover, \[\raisebox{-5pt}{\includegraphics[height=.2in]{capor.pdf}} \circ \raisebox{-5pt}{\includegraphics[height=.2in]{cupor.pdf}}  = \raisebox{-5pt}{\includegraphics[height=.2in]{sph.pdf}} \ \ \ \text{and} \ \ \ \ \raisebox{-5pt}{\includegraphics[height=.2in]{capor.pdf}} \circ ( \,\frac{1}{2} \, \raisebox{-5pt}{\includegraphics[height=.2in]{cupor-g1.pdf}} + \frac{1}{2} \,\raisebox{-5pt}{\includegraphics[height=.2in]{cupor.pdf}} \,) = \frac{1}{2}\, \raisebox{-3pt}{\includegraphics[height=.15in]{torus.pdf}} + \frac{h}{2} \,\raisebox{-5pt}{\includegraphics[height=.2in]{sph.pdf}} \]

which corresponds to $\epsilon(1) = 0, \epsilon(X) = 1.$ Below we recover the rules for $\Delta.$

\[ \raisebox{-8pt}{\includegraphics[height=.3in]{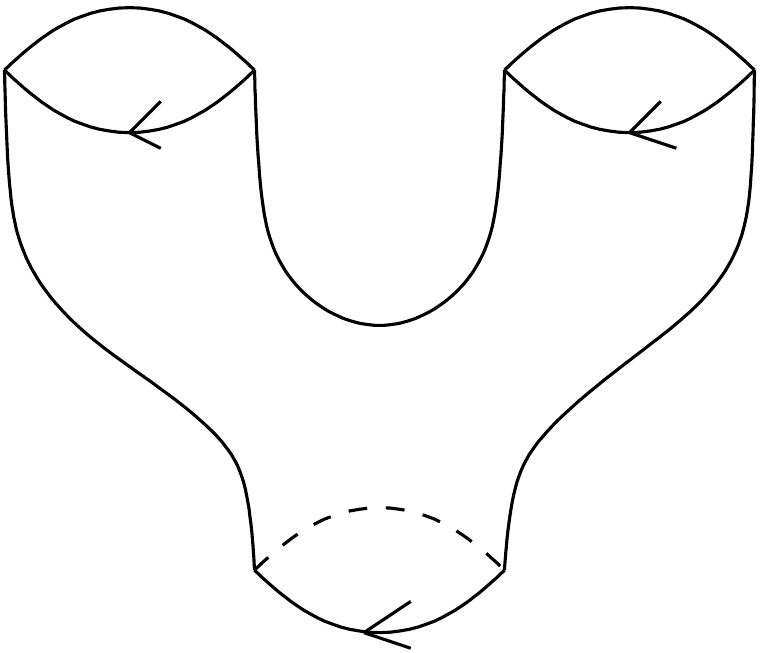}} \circ \raisebox{-5pt}{\includegraphics[height=.2in]{cupor.pdf}} \stackrel{(\tilde{SF})}{=} \frac{1}{2} \,\raisebox{-5pt}{\includegraphics[height=.2in]{cupor-g1.pdf}} \ \ \raisebox{-5pt}{\includegraphics[height=.2in]{cupor.pdf}} +  \frac{1}{2} \,\raisebox{-5pt}{\includegraphics[height=.2in]{cupor.pdf}} \ \ \raisebox{-5pt}{\includegraphics[height=.2in]{cupor-g1.pdf}} = (\,\frac{1}{2}\,\raisebox{-5pt}{\includegraphics[height=.2in]{cupor-g1.pdf}} + \frac{h}{2}\, \raisebox{-5pt}{\includegraphics[height=.2in]{cupor.pdf}} \,)\,\raisebox{-5pt}{\includegraphics[height=.2in]{cupor.pdf}}  + \raisebox{-5pt}{\includegraphics[height=.2in]{cupor.pdf}}\,  (\,\frac{1}{2}\,\raisebox{-5pt}{\includegraphics[height=.2in]{cupor-g1.pdf}} + \frac{h}{2}\, \raisebox{-5pt}{\includegraphics[height=.2in]{cupor.pdf}}  \,) - h \, \raisebox{-5pt}{\includegraphics[height=.2in]{cupor.pdf}} \, \raisebox{-5pt}{\includegraphics[height=.2in]{cupor.pdf}}\]

\[ \raisebox{-8pt}{\includegraphics[height=.3in]{upsidedown-pants.pdf}} \circ (\,\frac{1}{2}\,\raisebox{-5pt}{\includegraphics[height=.2in]{cupor-g1.pdf}} + \frac{h}{2} \,\raisebox{-5pt}{\includegraphics[height=.2in]{cupor.pdf}} \,)  \stackrel{(\tilde{SF}), (G2)}{=} (\, \frac{1}{2}\,\raisebox{-5pt}{\includegraphics[height=.2in]{cupor-g1.pdf}} + \frac{h}{2} \,\raisebox{-5pt}{\includegraphics[height=.2in]{cupor.pdf}}\,) \, (\,\frac{1}{2}\,\raisebox{-5pt}{\includegraphics[height=.2in]{cupor-g1.pdf}} + \frac{h}{2} \,\raisebox{-5pt}{\includegraphics[height=.2in]{cupor.pdf}} \,) + a \,\raisebox{-5pt}{\includegraphics[height=.2in]{cupor.pdf}} \,\raisebox{-5pt}{\includegraphics[height=.2in]{cupor.pdf}} \]

These say that $\Delta (1) = X \otimes 1 + 1 \otimes X - h\, 1 \otimes 1,\, \Delta (X) = X \otimes X + a\, 1 \otimes 1$ respectively.

Everything we did with $\textit{Foams}_{/\ell}$ can be appropriately modified for $\textit{Foams}_{/\tilde{\ell}},$ including the invariance under Reidemeister moves and functoriality property.

\end{document}